\documentclass[a4paper,11pt]{amsart}
\usepackage[latin1]{inputenc}
\usepackage{xcolor}

\usepackage{multicol}
\usepackage{amsfonts, amssymb, amsmath, amsthm, amscd,epsfig,mathrsfs, euscript}
\usepackage{latexsym}
\usepackage{graphicx}

\usepackage{cancel}
\usepackage[normalem]{ulem}

\usepackage[all]{xy}

\usepackage{enumerate}

\usepackage{verbatim}

\usepackage{hyperref}

\newtheorem{theorem}{Theorem}[section]

\newtheorem{corollary}[theorem]{Corollary}

\newtheorem{lemma}[theorem]{Lemma}

\newtheorem{proposition}[theorem]{Proposition}

{\par\noindent{\bf Proposition \ref{res:hiper}.}\!\!
	\nopagebreak\normalsize\it}%

{\par\noindent{\bf Theorem \ref{result43}.}\!\!
	\nopagebreak\normalsize\it}%

{\par\noindent{\it Sketch of the proof}.  
	\nopagebreak\normalsize}%
{\hfill\linebreak[2]\hspace*{\fill}$\circlearrowleft$}

{\par\noindent{\it Proof of Proposition }\ref{prop:stab:smc}.  
	\nopagebreak\normalsize}%
{\hfill\linebreak[2]\hspace*{\fill}$\circlearrowleft$}

{\par\noindent{\it Proof of Propositions }\ref{adap:mon}{\it and }\ref{simult:adap}.\!\!\!
	\nopagebreak\normalsize}%
{\hfill\linebreak[2]\hspace*{\fill}$\circlearrowleft$}

{\theoremstyle{definition}
	\newtheorem{definition}[theorem]{Definition}

	\newtheorem{remark}[theorem]{Remark}
	\newtheorem{example}[theorem]{Example}
	
	\newtheorem{parrafo}[theorem]{{\!}}}

\numberwithin{equation}{section}

	\DeclareMathOperator{\ord}{ord}

	\DeclareMathOperator{\Spec}{Spec}
	\DeclareMathOperator{\MaxSpec}{MaxSpec}

	\DeclareMathOperator{\Hord}{H-ord}
	\DeclareMathOperator{\Slaux}{-sl}
	\DeclareMathOperator{\In}{In}
	\DeclareMathOperator{\Gr}{Gr}

	\DeclareMathOperator{\wwin}{w-in} 
	
	\newcommand{\q}{{\mathfrak q}}
	\newcommand{\p}{{\mathfrak p}}

	\newcommand{\m}{{\mathfrak {m}}}

	\newcommand{\SSl}{\mathcal{S}\!\Slaux}

	\DeclareMathAlphabet{\mathpzc}{OT1}{pzc}{m}{it}
	\newcommand{\nub}{\overline{\nu}}

	\title{On some properties of the asymptotic Samuel function}
	\author{A. Bravo,  S. Encinas, J. Guill\'an-Rial}
\thanks{
The first two authors were partially supported by PID2022-138916NB-I00 funded by MCIN/AEI/10.13039/501100011033 and by ERDF A way of making Europe;
The first and third authors  were partially  supported from the Spanish Ministry of Science and Innovation, through the ``Severo Ochoa'' Programme for Centres of Excellence in R\&D (CEX2019-000904-S).
The first author was partially supported by the Madrid Government (Comunidad de Madrid - Spain) under the multiannual Agreement with UAM in the line for the Excellence of the University Research Staff in the context of the V PRICIT (Regional Programme of Research and Technological Innovation) 2022-2024. The third author was supported by the Spanish State Research Agency, through the Severo Ochoa and Mar\'ia de Maeztu Program for Centers and Units of Excellence in R\&D (CEX2020-001084-M). The third author also thanks CERCA Programme/Generalitat de Catalunya for institutional support.}
	

	\AtEndDocument{\bigskip{\footnotesize
			\textsc{Depto. Matem\'aticas, Facultad de Ciencias, Universidad Aut\'onoma de Madrid and ICMAT, Instituto de Ciencias Matem\'aticas CSIC-UAM-UC3M-UCM, Campus de Cantoblanco, 28049 Madrid, Spain} \newline
			\textit{E-mail address}, A. Bravo: \texttt{ana.bravo@uam.es}
			\medskip
			
			\noindent\textsc{Depto. \'Algebra, An\'alisis Matem\'atico, Geometr\'{\i}a y Topolog\'{\i}a, and IMUVA, Instituto de Matem\'aticas.	Universidad de Valladolid, Campus Miguel Delibes, 47011 Valladolid, Spain} \newline
			\textit{E-mail address}, S. Encinas: \texttt{santiago.encinas@uva.es}
			\medskip
			
			\noindent\textsc{Centre de Recerca Matem\`atica, Edifici C, Campus Bellaterra, 08193 Bellaterra, Spain} \newline
			\textit{E-mail address}, J. Guill\'an-Rial: \texttt{jguillan@crm.cat}
			\medskip
	}}
	
\begin{document}
		
\begin{abstract}
The asymptotic Samuel function generalizes to arbitrary rings the usual order function of a regular local   ring. Here we explore some natural properties in the context of excellent, equidimensional rings containing a field. In addition, we establish some results regarding the  Samuel slope of a local ring. This is an invariant related with algorithmic resolution of singularities of algebraic varieties. Among other results, we study its behavior after certain faithfully flat extensions. 
\end{abstract}
		
		\maketitle		

\section{Introduction}

Let $X$ be an algebraic variety over a field $k$. To give a constructive resolution of singularities of $X$ means to describe a procedure to construct the centers of a finite sequence of blow ups at regular centers, 
\begin{equation}
	\label{resolucion} 
	X_0=X \leftarrow X_1\leftarrow \ldots \leftarrow X_n
\end{equation}
so that $X_n$ is regular. This is usually accomplished (when known to exist) by defining some upper-continuous functions 
$$\Gamma_i: X_i\to (\Lambda, \leq),$$
where $(\Lambda, \leq)$ is some well ordered set, and where the   maximum value of $\Gamma_i$  determines the center of the monoidal transformation   $\pi_i: X_i\to X_{i-1}$, for $i=1,\ldots,n$. These {\em resolution functions} also provide us with a criterion to determine that the variety $X_{i+1}$ is {\em less singular} than $X_i$. 

\medskip

The construction of the functions  $\Gamma_i$ is somehow  involved, and yet, it  is strongly supported on the usual order function that one defines in a regular local ring. Furthermore,   it vastly exploits the nice properties of the  order function when defined in a smooth scheme of finite type over a perfect field. See for instance  the approach to resolution followed in \cite{EnVi} where this fact becomes quite evident.   

\medskip

\noindent{\bf Some properties of the order function in regular rings that play a key role in resolution} 

\medskip

Let $S$ be a regular ring and let $\mathfrak{q}\subset S$ be a prime ideal. 
The usual order of an element $s\in S$ at ${\mathfrak q}$ is
$$\nu_{\mathfrak{q}S_{\mathfrak{q}}}(s):=\max \{\ell: s\in {\mathfrak q}^{\ell}S_{\mathfrak{q}} \}.$$

{\bf (A)} The function  $\nu_{\mathfrak{q}S_{\mathfrak{q}}}$ is a valuation, and therefore for 
$a,b\in S$,
$\nu_{\mathfrak{q}S_{\mathfrak{q}}}(ab)=\nu_{\mathfrak{q}S_{\mathfrak{q}}}(a)+\nu_{\mathfrak{q}S_{\mathfrak{q}}}(b)$.

\medskip

{\bf (B)}  When $S$ is essentially of finite type over a perfect field $k$, for a fixed $s\in S$ the  function
$$\begin{array}{rrcl}
	\nu(s): & \Spec(S) & \longrightarrow & \mathbb{N}\cup\{\infty\} \\
	& \mathfrak{q} & \mapsto & \nu_{\mathfrak{q}S_{\mathfrak{q}}}(s)
\end{array}$$ $\nu(s)$ is upper semicontinuous. In particular,
for $\mathfrak{q}_1\subset\mathfrak{q}_2$
\begin{equation}
	\label{desigualdad_orden} 
	\nu_{{\mathfrak q}_1S_{\mathfrak{q}_1}}(s)\leq
	\nu_{{\mathfrak q}_2S_{\mathfrak{q}_2}}(s).
\end{equation}
Actually, the inequality (\ref{desigualdad_orden}) holds  for regular rings in general   (see \cite{DDGH}), and it can also be read in terms of the symbolic powers of ${\mathfrak q}_i$, namely, for all $\ell\in {\mathbb N}$, 
\begin{equation}
	\label{inclusion_potencias_simbolicas} 
	{\mathfrak q}_1^{(\ell)}\subseteq {\mathfrak q}_2^{(\ell)}. 
\end{equation}

{\bf (C)} When $S$ contains a field, and ${\mathfrak q}$ defines a regular subscheme in $\Spec(S)$, i.e., if  $S/{\mathfrak q}$ is regular, then  the ordinary and the symbolic powers of ${\mathfrak q}$ coincide,   
\begin{equation}
	\label{igualdad_symbolic_ordinary} 
	{\mathfrak q}^{(\ell)}={\mathfrak q}^\ell,
\end{equation}  
or in other words, 
for all $s\in S$,  
\begin{equation}
	\label{igualdad_primo_regular}
	\nu_{\mathfrak q}(s)=\nu_{{\mathfrak q}S_{\mathfrak q}}(s).
\end{equation} 
  This last property plays a special role, for instance,  to control the transforms of the resolution invariants after each of  the blow ups at the regular centers  in sequence (\ref{resolucion}).   

\medskip

\noindent{\bf The order function is used to define the resolution functions}

\medskip

Let us start by considering a special case. Let $S$ be a smooth  $k$-algebra of finite type over a perfect field $k$,  let $f(Z)\in S[Z]$ be a polynomial defining a hypersurface $X$ of maximum multiplicity $m>1$ in $\Spec(S[Z])$, and suppose we can write, 
\begin{equation}
	\label{polinomio} 
	f(Z)=Z^{m}+a_1Z^{m-1}+\ldots+a_m.
\end{equation}
  Already the order   stratifies the singularities of $f$ into locally closed strata.  Thus, to approach a resolution of $X$, one may think that the problem can be  reduced to  lowering the maximum order of a strict transform of $f$ below $m$. This is usually referred to as {\em resolving the pair $(\langle f\rangle, m)$}.
However,  just this information might not be enough to construct the resolution function $\Gamma: X\to (\Lambda, \geq)$. In particular, resolving the pair $(\langle f\rangle, m)$ requires the definition of new   functions, usually referred as    {\em resolution invariants}.  And, again, the main  source to defining them  relies, in one way or another, on  some order function of a  suitably defined local regular ring.  Thus,   $\Gamma$   would look something like this: 
\begin{equation}
	\label{funcion_resolucion}
\begin{array}{rrcl} 
	\Gamma: & X & \longrightarrow & \Lambda:= {\mathbb N}\times {\mathbb Q}_{>0}\times \ldots \times {\mathbb Q}_{>0}\\
	& \xi & \mapsto & (m, h_1(\xi), h_2(\xi),\, \dots\, ,  h_{\ell}(\xi)), 
\end{array}	
\end{equation}
where the set $\Lambda$ is ordered lexicographically. In particular this  means that if $\xi, \eta\in X$ are two points with the same multiplicity and if $\xi\in \overline{\eta}$, then, 
\begin{equation}
	\label{comparacion_h_1}
	h_1(\eta)\leq h_1(\xi).
	\end{equation}  
	 The value of $h_1(\xi)$ is the {\em weighted-order} at $\xi$ of   some ideal $J\subset S$ that collects information    coming from the coefficients of $f$. There are different strategies to define $h_1$: the so called {\em $\delta$-invariant}  coming from  {\em Hironaka's polyhedron of the singularity},  the order of the {\em coefficient ideal of the pair $(\langle f\rangle, m)$}, or the order of an  {\em elimination algebra of $(\langle f\rangle, m)$}, $\ord^{(d)}_X$, or the function  $\Hord_X$,  among others (see \cite{Benito_V}, \cite{Benito_V_Compositio}, \cite{BVIndiana},  \cite{Br_V1}, \cite{CJS2020}, \cite{CJSch2019}, \cite{HironakaPoliedro}, \cite{Schober}, \cite{Villa_Advances}). The rest of the functions $h_i$, $i>1$, depend in some sense on the construction of $h_1$. 

\medskip

The previous example covers the case of a hypersurface, since the defining equation $f$ of a hypersurface $X$ can be assumed to have the form in (\ref{polinomio}) after choosing a suitable local  (\'etale) embedding in a neighborhood of a singular point. In addition, the case of an arbitrary variety can be  {\em reduced to the hypersurface case}, also, after considering a suitable local (\'etale) embedding (see   \cite{HironakaIdealistic}, or   \cite{V}). Thus our initial example already gives us a rough picture  of a procedure to construct the function  $\Gamma$.

\medskip 

Notice that the definition of   the resolution functions  strongly uses local-\'etale embeddings of $X$ into smooth ambient spaces, where the good properties of the usual order function come in handy.  As a counterpart, some (non-trivial) work has to be done in order to show   that the resolution functions are   independent of the   embeddings. This is needed to prove, for instance,    that  the centers to blow up in sequence (\ref{resolucion}), which  are   determined locally, patch as to define global centers on $X$ that ultimately lead   to  a resolution of singularities of $X$. 
And sometimes the use of the \'etale topology is not enough. For instance, the invariants provided by Hironaka's polyhedron are constructed at the   completion of a local regular ring  where the ideal of the variety is defined.  In this line, we should mention the works of Cossart-Piltant in  \cite{CosPil2015} and  Cossart-Schober in \cite{CosSch2021},  where  it is shown that to construct the Hironaka's polyhedron,  the completion can be avoided. 

\medskip

We can go one step further and   explore  properties of the local rings at the singular points of $X$ that allow us to collect information regarding the  resolution functions:  Can we avoid the use of a local embedding in a regular ring and get information directly from the singular local ring of a variety?

\medskip

When $\xi\in X$ is a singular point, it is still possible to consider the order function at ${\mathcal O}_{X,\xi}$, but  this does not behave very nicely. To start with, it is far from being  upper semicontinuous. 
A function that has a much nicer behavior is the {\em asymptotic Samuel function}. 

\medskip 

In this paper we study some properties of the asymptotic Samuel function which could be useful to understand resolution functions from an intrinsic point of view. In this sense, we approach two different problems. On the one hand, we explore the properties of the asymptotic Samuel function in comparison to the properties {\bf (A), (B)} and {\bf (C)} listed before. In particular we will see that {\bf (B)} and {\bf (C)} hold for equimultiple primes (compare to the inequality (\ref{comparacion_h_1})).  On the other hand, we continue the work started in \cite{BeBrEn}, where we used the asymptotic Samuel function to define  an invariant, the {\em Samuel slope} of a local ring,  which can be defined for any local Noetherian ring.  In \cite{BeBrEn} we showed that the Samuel slope is  connected to some resolution functions that appear naturally when working with algebraic varieties over perfect fields of prime characteristic. In particular, there seems to be a strong  connection between the Samuel slope of a singular point   and the value of the function $h_1$ mentioned above (see \cite{BeBrEn2}).  Here  we do not restrict to algebraic varieties and    explore further properties of this invariant in the wider context of excellent local equicharacteristic equidimensional rings, and, among other results, we prove inequalities in the line of (\ref{comparacion_h_1}) when comparing the Samuel slope at equimultiple primes. 

\bigskip

\noindent{\bf Some definitions and main results}

\medskip

\begin{definition}\label{definicion_Samuel} 
 	Let  $A$ be a Noetherian ring and let  $I\subset A$ be a proper ideal.  The {\em asymptotic Samuel function at $I$}, $\bar{\nu}_I:A\to\mathbb{R}\cup\{\infty\}$, is defined as: 
	\begin{equation}\label{Def_Sam_lim}
		\bar{\nu}_I(f)=\lim_{n\to\infty}\frac{\nu_I(f^n)}{n}, \qquad f\in A.
	\end{equation}
\end{definition}

This function was first introduced by Samuel in \cite{Samuel}, when studying the behavior of powers of ideals.  Afterwards,   Rees pursued the use of this function in \cite{Rees1955}, \cite{Rees_Local}  where it is shown that the limit exists, see also \cite[Lemma 6.9.2]{Hu_Sw},  \cite{Rees_Ideals}, \cite{Rees_Ideals_II}. If  $(A,{\mathfrak m})$ is a local regular ring, then $\nub_{\mathfrak m}$ is   the ordinary order function at the maximal ideal of $A$ (and then we write $\nu_{\mathfrak m}$). 

\medskip

The asymptotic Samuel function  measures how deep a given element lies into the integral closure of an ideal, i.e., 
\begin{equation}
	\label{nub_integral}
	\bar{\nu}_I(f)\geq \frac{n}{\ell} \Longleftrightarrow f^{\ell}\in\overline{I^n},
\end{equation}
see \cite[Corollary 6.9.1]{Hu_Sw}.   If      $I$ is  not contained in a minimal prime of $A$ and $\{v_1,\ldots,v_s\}$ is  a set of Rees valuations of  $I$, then
\begin{equation}
	\label{Samuel_valorativo}
	\bar{\nu}_I(f)=\min\left\{\frac{v_i(f)}{v_i(I)}\mid i=1,\ldots,s\right\}, 
\end{equation}
(see \cite[Lemma 10.1.5, Theorem 10.2.2]{Hu_Sw} and \cite[Proposition 2.2]{Irena}). Therefore, if $f\in A$ is not nilpotent, $\nub_I(f)\in {\mathbb Q}$.  

{\begin{remark} \label{properties_order_function} The asymptotic Samuel function is an order function. It can be checked that the following hold: 
		\begin{enumerate}
		\item[(i)] for $f,g\in A$, $\nub_I(f+g)\geq \min\{\nub_I(f), \nub_I(g)\}$, with equality if $\nub_I(f)\neq\nub_I(g)$;
			\item[(ii)]  for $f,g\in A$, $\nub_I(fg)\geq  \nub_I(f)+ \nub_I(g)$.
		\end{enumerate}
In addition, it is worthwhile noticing that for $f\in A$ and $\ell\in {\mathbb N}$, $\nub_I(f^{\ell})=\ell\cdot \nub_I(f)$. 		  
\end{remark}}
	
\medskip

Definition \ref{definicion_Samuel}  can be extended to the case in which arbitrary filtrations of ideals are considered. This  has been studied by Cutkosky and Praharaj in  \cite{CutPraj}.  On the other hand,  we refer to the work of Hickel in   \cite{Hickel13} for some results on the explicit computation of the asymptotic Samuel function on complete local rings. Some of these results will play a role in our arguments, and they will be precisely stated and properly referred in section \ref{seccion_finite_transversal}.  

\medskip

\begin{parrafo}
{\bf Properties (A), (B), (C)} 
\end{parrafo}

In the following lines we will revisit properties \textbf{(A)},$\,$\textbf{(B)} and \textbf{(C)} in the case of the asymptotic Samuel function.
\medskip

{\bf (A)} In general, the asymptotic Samuel function is not a valuation and it is not hard to find examples.
\begin{example}
Let $k$ be a field.
Consider the ring $B=k[x,y,z]/\langle xy-z^3\rangle$, with maximal ideal 
${\mathfrak m}=\langle \overline{x},\overline{y},\overline{z}\rangle$. 
We have that:
             
$$\nub_{\mathfrak m}(\overline{x})=\nub_{\mathfrak m}(\overline{y})=\nub_{\mathfrak m}(\overline{z})=1;$$
$$\nub_{\mathfrak m}(\overline{xy})=3; \  \  \nub_{\mathfrak m}(\overline{xz})=2;  \  \ \nub_{\mathfrak m}(\overline{yz})=2.$$ 
\end{example}

Here, it can be checked that $\langle \overline{x},\overline{y} \rangle$ is not a reduction of $\m$. However, there are minimal reductions of $\m$ which contain the element $\overline{z}$. In fact,  
under some assumptions, using minimal reductions of $\m$, one can identify a regular subring of $B$
where the restriction of $\nub_{\mathfrak{m}}$ behaves as a valuation.
The following result clarifies what is going on in the previous example, in fact,
a little bit more can be stated:
\medskip

\noindent \textbf{Proposition \ref{MiniMax}}
{\em Let $(B,{\mathfrak m})$ be an equidimensional excellent equicharacteristic local ring of dimension $d\geq 1$ containing  a field $k$. Suppose ${\mathfrak m}$ has a reduction generated by $d$ elements, $y_1,\ldots, y_d\in {\mathfrak m}$. Set $A:=k[y_1,\ldots, y_d]_{\langle y_1,\ldots, y_d\rangle}\subset B$.
Then, for $a\in A$ and $b\in B$, 
	$$\nub_{\mathfrak m}(ab)=\nub_{\mathfrak m}(a)+  \nub_{\mathfrak m}(b).$$
}

\medskip

{\bf (B), (C)} Let $B$ be a  non-necessarily regular ring. For our discussions,  the case $\dim(B)=0$ can be left out, thus through the paper we will be assuming that $\dim(B)\geq 1$.  For a prime ideal  ${\mathfrak p}\subset B$, using  the asymptotic Samuel function we can  define the following filtrations:  for $r\in {\mathbb Q}_{>0}$,   
$${\mathfrak p}^{\geq r} :=\{b\in B: \nub_{\mathfrak p}(b)\geq r\}, \qquad {\mathfrak p}^{(\geq r)} :=\{b\in B: \nub_{{\mathfrak p}B_{\mathfrak p}}(b)\geq  r\}.$$

In general, we do not expect that  properties {\bf (B)} and {\bf (C)} hold: 
\begin{example}
	Let $k$ be a field, and  let $B=\left(k[x,y,z]/\langle y^2 + zx^3\rangle\right)_{\m}$, where ${\mathfrak m}=\langle \overline{x},\overline{y},\overline{z}\rangle$. Set ${\mathfrak p}=\langle \overline{y},\overline{z} \rangle$. Notice that 	
	$$\nub_{\mathfrak m}(\overline{z})=1 \leq 	  \nub_{{\mathfrak p}B_{\mathfrak p}}(\overline{z})=2.$$ \\
	In addition,  ${\mathfrak p}$ defines a regular prime in $B$, i.e. $B/\mathfrak{p}$ is regular,  however, 
	$$1=\nub_{\mathfrak p}(\overline{z})< \nub_{{\mathfrak p}B_{\mathfrak p}}(\overline{z})=2.$$ 
\end{example}
As indicated before,  we will see that, to expect similar properties as in the regular case, we have to   restrict  to  primes with the same multiplicity. In other words, the asymptotic Samuel function {\em behaves as expected} when restricted to a  (locally closed) stratum   of constant multiplicity  of $\Spec(B)$.   
 
 \medskip 
 
To fix notation, for a prime ${\mathfrak p}$  in $\Spec({B})$ and for a ${\mathfrak p}$-primary ideal, ${\mathfrak a}\subset B$, we will use $e_{B_{\mathfrak p}}({\mathfrak a}B_{\mathfrak p})$ to denote the multiplicity of the local ring $B_{\mathfrak p}$ at  ${\mathfrak a}B_{\mathfrak p}$. Now properties {\bf (B)} and {\bf (C)} have the following reformulation in the context of singular rings: 

\medskip

\noindent{\bf Theorem \ref{casi_semi_continuo}}
{\em Let $B$ be an equidimensional excellent ring containing a field.
Let ${\mathfrak p}_1\subset{\mathfrak p}_2\subset B$ be two prime ideals such that   $e_{B_{\mathfrak{p}_1}}({\mathfrak p}_1B_{\mathfrak{p}_1})=
e_{B_{\mathfrak{p}_2}}({\mathfrak p}_2B_{\mathfrak{p}_2})$.
Then $\nub_{{\mathfrak p}_1B_{\mathfrak{p}_1}}(b)\leq
\nub_{{\mathfrak p}_2B_{\mathfrak{p}_2}}(b)$ for $b\in B$. }

\medskip

In particular, this says that for ${\mathfrak p}_1\subset{\mathfrak p}_2$ as in the theorem, and $r\in\mathbb{Q}_{>0}$,
$${\mathfrak p}_1^{(\geq r)}\subseteq {\mathfrak p}_2^{(\geq r)}.$$

\medskip 	

\noindent {\bf Theorem \ref{casi_semi_continuo_sin_localizar}}
{\em  Let $B$ be an equidimensional excellent ring containing a field.
Let ${\mathfrak p}\subset B$ be a prime in the top multiplicity locus of $B$,
and assume that $B/{\mathfrak p}$ is regular. Then 
$\nub_{\mathfrak p}(b)=\nub_{{\mathfrak p}B_{\mathfrak p}}(b)$ for $b\in B$. }

\medskip

Hence, in particular, for ${\mathfrak p}$ as in the theorem,
and  $r\in\mathbb{Q}_{>0}$,
$${\mathfrak p}^{(\geq r)}={\mathfrak p}^{\geq r}.$$

\medskip

To conclude this part,  if $(S,\mathfrak{n})$ is a regular local ring, the usual order induces the natural filtration $\{{\mathfrak n}^{\ell}\}_{\ell \in {\mathbb N}}$, where
$${\mathfrak n}^{\ell}=\{s\in S: \nu_{\mathfrak n}(s)\geq \ell\},$$
which in turns leads to the usual graded ring 
$\text{Gr}_{\mathfrak n}(S)=\bigoplus\limits_{\ell \in {\mathbb N}}{\mathfrak n}^{\ell}/{\mathfrak n}^{\ell+1}$ that  is graded over ${\mathbb N}$ and finitely generated over $S/{\mathfrak n}$. For an arbitrary   local ring $(B,\mathfrak{m},k)$,
we can consider the graded ring associated to the filtration induced by the asymptotic Samuel function: setting 
		${\mathfrak m}^{>r}:=\{b\in B: \nub_{\mathfrak m}(b)>r\}$, define
		 $$\overline{\Gr}_{\mathfrak m}(B):= \bigoplus_{r\in\mathbb{Q}_{\geq 0}}{\mathfrak m}^{\geq r}/{\mathfrak m}^{> r}.$$
	Notice that by (\ref{Samuel_valorativo}),  there is some integer $n\in {\mathbb N}$ so that $\overline{\Gr}_{\mathfrak m}(B)$ is graded over $\frac{1}{n}{\mathbb N}$. And actually $n$ can be taken as $m!$  if $m=e_{B_{\text{red}}}({\mathfrak m}_{\text{red}})$, where $B_{\text{red}}$ denotes $B/\text{Nil(B)}$ and ${\mathfrak m}_{\text{red}}={\mathfrak m}/\text{Nil}(B)$ (see Remark \ref{factorial}). Imposing some (mild) conditions  on $B$, we can show that $\overline{\Gr}_{\mathfrak m}(B)$ is a  $k$-algebra of finite type:

		\medskip 
		
		\noindent
{\bf Theorem \ref{graduado_barra_finito}}
{\em  Let $(B,{\mathfrak m},k)$ be an excellent local ring. Then $\overline{\Gr}_{\mathfrak m}(B)$ is a $k$-algebra of finite type.
 }

\begin{parrafo}\label{DefSamSlope}
{\bf The Samuel slope of a local ring}	 
\end{parrafo}

\noindent 
The notion of {\em Samuel slope of a local ring} was introduced in \cite{BeBrEn}. More precisely, for  a Noetherian local ring $(B, \mathfrak{m},k)$ of dimension $d$ we consider the natural map:
$$\lambda_{\mathfrak{m}}: {\mathfrak m}/ {\mathfrak m}^2 \to {\mathfrak m}^{\geq 1}/{\mathfrak m}^{> 1}. $$
  If  $(B, \mathfrak{m}, k)$ is not regular, then  $\ker(\lambda_{\mathfrak{m}})$   might be non trivial, and
 its  dimension as a $k$-vector space is an invariant of the ring. To start with, it can be proven that $\dim_k\left(\ker(\lambda_{\mathfrak{m}}) \right)$   is bounded above by the {\em excess of embedding dimension of $B$},
$\text{exc-emb-dim}(B)$, i.e., 
$$0\leq \dim_{k}\left(\ker(\lambda_{\mathfrak{m}})\right) \leq \text{exc-emb-dim}(B):= \dim_{k}\left({\mathfrak m}/ {\mathfrak m}^2\right)-\dim(B).$$
This follows from the fact that elements $a\in \ker(\lambda_{\mathfrak m})$ are nilpotent in $\text{Gr}_{\mathfrak m}(B)$.
\medskip

Local non-regular rings where the upper bound is not achieved seem to have milder singularities than the others, and in such case we say that  {\em Samuel slope of $B$, $\SSl(B)$},   is 1. 
\medskip

When $\dim_{k}(\ker(\lambda_{\mathfrak{m}})) =  \text{exc-emb-dim}(B)>0$  we say that  $B$ is  in the {\em extremal case}, and  we define  the Samuel slope   as follows.
\medskip

Set  $t=\text{exc-emb-dim}(B)>0$.
We   say that the elements $\gamma_1,\ldots, \gamma_t\in {\mathfrak m}$ form a  {\em $\lambda_{\mathfrak{m}}$-sequence} if their classes in
${\mathfrak m}/ {\mathfrak m}^2$ form a basis of $\ker(\lambda_{\mathfrak{m}})$.
Then, 
\begin{equation*}
\SSl(B):=
\sup\limits_{{\lambda_{\mathfrak{m}}}{\text{-sequence}}} \left\{
\min\left\{\bar{\nu}_{\mathfrak{m}}(\gamma_1),\ldots,\bar{\nu}_{\mathfrak{m}}(\gamma_t)\right\} 
\right\}\in\mathbb{R}_{\geq 0}\cup\{\infty\} ,
\end{equation*}
where the supremum is taken over all the $\lambda_{\mathfrak{m}}$-sequences of $B$. 

\medskip

Equivalently, $\SSl(B)$ can also defined in the  following way. Let $\mathbf{x}=\{x_1,\ldots,x_{d+t}\}\subset\mathfrak{m}$ be a minimal set of generators of $\mathfrak{m}$.
	We define the \emph{slope with respect to $\mathbf{x}$}  as
	$$\text{sl}_{\mathbf{x}}(B):=\min\{\nub_{\mathfrak{m}}(x_{d+1}),\ldots,\nub_{\mathfrak{m}}(x_{d+t})\}.$$
	And then,   
	\begin{equation*}
		\SSl(B):=
		\sup\limits_{\mathbf{x}}\text{sl}_{\mathbf{x}}(B)=
		\sup\limits_{\mathbf{x}} \left\{
		\min\left\{\bar{\nu}_{\mathfrak{m}}(x_{d+1}),\ldots,\bar{\nu}_{\mathfrak{m}}(x_{d+t})\right\} 
		\right\},
	\end{equation*}
	where the supremum is taken over all possible minimal  ordered sets   of generators $\mathbf{x}$ of $\mathfrak{m}$. To conclude,  if $(B,{\mathfrak m})$ is regular we set $\SSl(B):=\infty$. 
\medskip 

\begin{example}
Let $k$ be a field of characteristic 2.
Set $B=\left(k[x,y]/\langle x^2+y^4+y^5 \rangle\right)_{\mathfrak{m}}$, where
$\mathfrak{m}=\langle \overline{x},\overline{y}\rangle$.
Then $(B,\mathfrak{m})$ is in the extremal case and
$\ker(\lambda_{\mathfrak{m}})=\langle \In_{\mathfrak{m}}(\overline{x})\rangle$.
Both $\{\overline{x}\}$ and $\{\overline{x}+\overline{y}^2\}$ are 
$\lambda_{\mathfrak{m}}$-sequences.
However $\nub_{\mathfrak{m}}(\overline{x})=2$, while 
$\nub_{\mathfrak{m}}(\overline{x}+\overline{y}^2)=5/2$.
In fact $\SSl(B)=5/2$, see Corollary \ref{noentero}.
\end{example}

Here we also study the Samuel slope in the wider setting of equicharacterictic equidimensional excellent local rings.  First of all, just from the definition, it is not clear that this invariant is finite in the case of non-regular rings. We show:
\medskip

\noindent {\bf Theorem \ref{main_theorem}} {\em  
Let $(B,{\mathfrak m},k)$ be a
non-regular reduced equicharacteristic equidimensional excellent local ring.  
Then $\SSl(B)\in\mathbb{Q}$.  
}

\medskip
Next we consider the case of non-reduced rings. Recall that the multiplicity induces the same stratification on both, $B$ and $B_{\text{red}}$. We show that both rings share the same Samuel slope.

\medskip

\noindent {\bf Theorem \ref{ThSlopeReduced}}
{\em Let $(B,{\mathfrak m},k)$ be a non-reduced equicharacteristic equidimensional  excellent local ring. Then $\SSl(B)=\SSl(B_{\text{\tiny{\rm red}}})$.  
}
\medskip

\noindent From here it follows that $\SSl(B)=\infty$ if and only if
$B_{\text{\tiny{\rm red}}}$ is a regular local ring (Corollary \ref{Pendienteinfinita}).
\medskip

Since the Samuel slope is an invariant of the local ring of a singularity, one would expect that it be preserved under \'etale extensions and completion. In \cite{BeBrEn} it was shown that  $(B',{\mathfrak m}')$ is a local \'etale extension of $(B,{\mathfrak m})$  with the same residue field, then $\SSl(B)=\SSl({B}')$. The argument given there also  shows that $\SSl(B)=\SSl(\widehat{B})$, where $\widehat{B}$ is the ${\mathfrak m}$-adic completion of $B$.  Here we treat the case of  arbitrary local \'etale extensions, which requires a different strategy. 

\medskip

\noindent{\bf Theorem \ref{ThSlopeEtale}} 	{\em  Let $(B,{\mathfrak m},k)$ be an  equicharacteristic equidimensional excellent local ring,   and let  $(B,{\mathfrak m})\to (B',{\mathfrak m}')$ be a local-\'etale extension. Then $\SSl(B)=\SSl(B')$.  }
\medskip 

To conclude, as we mentioned the Samuel slope of a local ring seems to be connected to the resolution invariant $h_1$ in (\ref{funcion_resolucion}) which has the following property:   if   ${\mathfrak p}\subset { \mathfrak m}$ is an equimultiple prime then 
$$h_1(\mathfrak p) \leq h_1({ \mathfrak m}).$$
In fact, when the characteristic is zero, then $h_1$ is upper semi-continuous when restricted to points with the same multiplicity.  For a Noetherian ring $B$ we can consider the function
$$\begin{array}{rrcl}
\SSl: & \Spec(B) & \longrightarrow & \mathbb{Q}\cup\{\infty\} \\
 & \mathfrak{p} & \mapsto & \SSl(B_{\mathfrak{p}}).
\end{array}$$
In general, $\SSl$ is not upper semicontinuous,  see Example \ref{Whitney}, but  it 
has the following nice property on the maximal spectrum, $\MaxSpec(B)$, of $B$:  
\medskip

\noindent{\bf Theorem \ref{ThSlopeSemicont}} {\em 
Let $B$ be an equidimensional equicharacteristic excellent ring   and let
$\mathfrak{p}\in\Spec(B)$.
Then there is a dense open set $U\subset\MaxSpec(B/\mathfrak{p})$ such that
$$\SSl(B_{\mathfrak p})\leq\SSl(B_{\mathfrak{m}})\qquad
\text{for all }\quad \mathfrak{m}/\mathfrak{p}\in U.$$
}

\medskip

The paper is organized as follows. One of the main ingredients in our proofs is the  use of the so called \textit{finite-transversal projections} together with Hickel's result on the computation of the asymptotic Samuel function. These are treated in section \ref{seccion_finite_transversal}, where we also address  Proposition \ref{MiniMax}. The proofs of Theorems \ref{casi_semi_continuo}, \ref{casi_semi_continuo_sin_localizar}, and
\ref{graduado_barra_finito} are addressed in section \ref{seccion_naturalprops}. The rest of the paper is dedicated to the Samuel slope of a local ring. Theorems \ref{main_theorem} and  \ref{ThSlopeReduced} are proved in section \ref{seccion_finitud}, while Theorem \ref{ThSlopeEtale} is proved in section \ref{SeccFFlatExt}. Finally, a proof of Theorem \ref{ThSlopeSemicont} is given in section \ref{seccion_comparing slopes at prime ideals}.
\medskip

{\em Acknowledgments.} We profited from conversations with A. Benito, C. del-Buey-de-Andr\'es and O. E. Villamayor. 
We also want to thank the referee for several useful suggestions and constructive 
criticism of the paper.

\section{Finite-transversal projections}\label{seccion_finite_transversal}

Finite-transversal projections were considered in \cite{V} for the construction of {\em local presentations of the multiplicity function }of algebraic varieties defined over a perfect field. The existence of such presentations implies that resolution of singularities of algebraic varieties can be achieved via successive simplifications of the multiplicity (in characteristic zero). Finite-transversal projections were further explored in \cite{COA}, where they were considered between (non-necessarily regular) algebraic varieties defined over perfect fields. Some other properties of such morphisms are discussed in \cite{BrEnHandbook}. In this section we treat this notion in a more general setting, dropping the assumption that the rings involved be  $k$-algebras of finite type.

\begin{definition} \label{DefTransv} 
	Let $S\subset B$ be a finite extension of excellent rings 
	with $S$ regular and $B$ equidimensional and reduced. Let $K$ be the fraction field of $S$ and let $L:=B\otimes_S K$. 
	Suppose that no non-zero element of $S$ is a zero divisor in $B$.
	We say that the projection $\Spec(B)\to\Spec(S)$ (or the extension $S\subset B$) is \emph{finite-transversal with respect to}
	${\mathfrak p}\in\Spec(B)$ if:
	$$e_{B_{{\mathfrak p}}}({\mathfrak p} B_{{\mathfrak p}})=[L:K].$$
If $(S,\mathfrak{n})\to(B,\mathfrak{m})$ is a finite-transversal extension
of local rings with respect to $\mathfrak{m}$ then we simply say that 
$(S,\mathfrak{n})\to(B,\mathfrak{m})$ is a \emph{finite-transversal extension}.
\end{definition}

Using  Zariski's formula for the multiplicity for finite projections, \cite[Theorem 24, page 297 and Corollary 1, page 299]{Z-SII}, one can get the following characterization of finite-transversal projections:  

\begin{proposition} \label{ZarCond}
	\cite[Corollary 4.9]{V} 
 Let $S\subset B$ be a finite extension of excellent rings  with $S$ regular, and  $B$  equidimensional and reduced. Suppose that no non-zero element of $S$ is a zero divisor in $B$.  	Let ${\mathfrak p}\subset B$ be a prime ideal, and let 
	$\mathfrak{q}=\mathfrak{p}\cap S$. 
	Then the following are equivalent:
	\begin{enumerate}
		\item $e_{B_{\mathfrak{p}}}(\mathfrak{p} B_\mathfrak{p})=[L:K]$.
		\item The following three conditions hold:
		\begin{itemize}
			\item[(i)] $\mathfrak{p}$ is the only prime of $B$ dominating $\mathfrak{q}$,
			\item[(ii)] $k(\mathfrak{p})=B_{\mathfrak{p}}/\mathfrak{p}B_{\mathfrak{p}}=S_{\mathfrak{q}}/\mathfrak{q}S_{\mathfrak{q}}=k(\mathfrak{q})$,
			\item[(iii)] $e_{B_{\mathfrak p}}({\mathfrak q}B_{\mathfrak p})=e_{B_{\mathfrak p}}({\mathfrak p}B_{\mathfrak p})$.
		\end{itemize}
	\end{enumerate}
\end{proposition} 

Observe that by Rees' Theorem,  condition (2) (iii) is equivalent to asking that 
	$\mathfrak{q} B_{\mathfrak{p}}$ be  a reduction of the ideal $\mathfrak{p}B_{\mathfrak{p}}$. To be able to use Rees' Theorem we will be assuming that $B$ is an excellent ring.

\medskip

\noindent{\bf {\em On finite-transversal morphisms and the   asymptotic Samuel function}}

\medskip 

\noindent Finite-transversal projections will play a central role in our arguments,
mainly because of the combination of  the outputs of 
Proposition \ref{CoeffenS},  due to Villamayor, and  a theorem of M. Hickel for the computation of the asymptotic Samuel function, Theorem \ref{Hickel_refinado} below.
In addition in \S\ref{setting_1} we briefly describe how to construct 
finite-transversal morphisms for some faithfully flat extension of a given ring $B$.
This will be frequently used in the rest of the paper.

\begin{proposition} \label{CoeffenS}
	\cite[Lemma 5.2]{V}
	Let $S\subset B$ be  a finite extension such that the non-zero elements of $S$ are non-zero divisors in $B$.  
	Assume that $S$ is a regular ring and let $K=K(S)$   be the quotient field of $S$. Let $\theta\in B$ and let $f(Z)\in K[Z]$   be  the monic polynomial of minimal degree such that $f(\theta)=0$.
	If $S[\theta]$ denotes the $S$-subalgebra of $B$ generated by $\theta$, then:
	\begin{enumerate}
		\item the coefficients of $f$ are in $S$, i.e. $f(Z)\in S[Z]$, and
		\item $S[\theta]\cong S[Z]/\langle f(Z)\rangle$.
	\end{enumerate}
\end{proposition}

\begin{theorem}\label{Hickel_refinado} \cite[Theorem 2.1]{Hickel13}
	Let $(B, {\mathfrak m})$ be a Noetherian equicharacteristic equidimensional and excellent local ring of Krull dimension $d$.
	Assume that there is a  a 
	faithfully flat extension  $(B, {\mathfrak m}) \to (\widetilde{B}, \widetilde{\mathfrak m})$  with ${\mathfrak m}\widetilde{B}=\widetilde{\mathfrak m}$ together with  a 
  finite-transversal morphism  with respect to $\widetilde{\mathfrak m}$,  $S\subset \widetilde{B}$.
	Let $b\in B$. If 
	$$p(Z)=Z^{\ell}+a_1Z^{\ell-1}+\ldots+a_{\ell}$$
	is the minimal polynomial of $b\in \widetilde{B}$ over the fraction field of $S$,  $K(S)$, then 
	$p(Z)\in S[Z]$ and 
	\begin{equation}
		\label{calculo_orden}
		\nub_{\mathfrak m}(b)=\nub_{\widetilde{\mathfrak m}}(b)=\min_i \left\{\frac{\nu_{{\mathfrak n}} (a_i)}{i}: i=1,\ldots,\ell\right\},
	\end{equation}
	where ${\mathfrak n}=\widetilde{\mathfrak m}\cap S$. 
\end{theorem}
\begin{proof}
Theorem 2.1 in \cite{Hickel13} is stated in the case in which $\widetilde{B}=	\widehat{B}$, and then a reduction to the domain case is considered.
See \cite[Theorem 11.6.8]{BrEnHandbook}, where it is checked that Hickel's Theorem
holds in this more general setting.
\end{proof}

\begin{parrafo} \label{setting_1} {\bf Constructing finite-transversal projections.} Given an  excellent reduced  equidimensional ring $B$, and  a point 
${\mathfrak p}\in\Spec(B)$, in general, there might not be  
a regular ring $S$ and  a finite-transversal projection with respect to ${\mathfrak p}$, $\Spec(B)\to\Spec(S)$. To start with, it is a  necessary condition  that ${\mathfrak p}$ has a reduction generated by $\dim(B_{\mathfrak p})$-elements. But even if such condition is satisfied, the existence of the required finite  projection is not guaranteed (see \cite[Example 11.3.11]{BrEnHandbook}). However, finite-transversal projections can be constructed if we are allowed to extend our ring $B$.    For instance, in \cite{V} (see \cite[Proposition 31.1]{Br_V2}) it is proven that if   $B$ is  essentially of finite type over a perfect field $k$, then a finite-transversal projection can be constructed in some \'etale extension of $B$.   
\medskip	

\noindent
\textbf{Existence of finite transversal projection.}
Suppose that $(B,{\mathfrak m})$ is a local equicharacteristic equidimensional excellent reduced ring. Assume also that $B$ contains  a reduction of ${\mathfrak m}$
generated by $d=\dim (B)$ elements, $x_1,\ldots,x_d\in B$.
Now, denote by $k'$ a coefficient field of the $\mathfrak{m}$-adic completion
$(\widehat{B},\widehat{\mathfrak{m}})$ and set
$S=k'[[x_1,\ldots,x_d]]\subset \widehat{B}$.
Since $x_1,\ldots,x_d$ are analytically independent, 
$S$ is a ring of power series in $d$ variables.
The extension 
\begin{equation} \label{ExtCompletado}
S=k'[[x_1,\ldots,x_d]]\subset \widehat{B}
\end{equation}
is finite by \cite[Theorem 8, page 68]{Cohen} and, moreover, finite-transversal with respect to  $\widehat{\mathfrak{m}}$.
See also \cite[Proof of Theorem 1.1]{Hickel13}.

The extension $B\to\widehat{B}$ is faithfully flat, this means that for any ideal $I\subset B$ and $b\in B$ we have
\begin{equation} \label{EqualComplete}
\nub_{I}(b)= \nub_{I\widehat{B}}(b).
\end{equation}

Note that if the residue field of $B$ is infinite then
by \cite[Proposition 8.3.7]{Hu_Sw}, $B$ contains a reduction of ${\mathfrak m}$
generated by $d=\dim (B)$ elements.
\medskip

\noindent
\textbf{Extension to the case with reduction with $d$ elements.}
If $B$ does not contain a reduction of ${\mathfrak m}$ generated by $d$ elements, then we want to produce a faithfully flat extension $(B,\mathfrak{m})\to (B_1,\mathfrak{m}_1)$ such that $\mathfrak{m}_1$ has such reduction.
We consider two possibilities as follows:
\begin{description}
\item[(a)] If $B$ contains a field $k$ consider a suitable \'etale extension of $k$, ${k}_1\supset k$, so that, after localizing at a maximal ideal
${\mathfrak m}_1\subset B\otimes_k k_1$, the local ring 
\begin{equation} \label{ExtEtale}
{B}_1:=(B\otimes_k k_1)_{{\mathfrak m}_1}
\end{equation} 
contains a reduction generated by $d$-elements.
\item[(b)] Other possibility is to set the ring
\begin{equation} \label{ExtBx}
B_1=(B[x])_{\mathfrak{m}[x]}
\end{equation}
which has infinite residue field.
\end{description}
\medskip

Note that in both cases we have that
\begin{equation} \label{EqualB1}
\nub_{I}(b)= \nub_{IB_1}(b),
\end{equation}
for any ideal $I\subset B$ and $b\in B$.
\medskip
	
\noindent
\textbf{Reduction to reduced rings.}
In general we will be dealing with a local ring $(B,{\mathfrak m})$ of Krull dimension $d\geq 1$.
And we will be interested in proving results  concerning the asymptotic Samuel function,   
	$\nub_{\mathfrak m}: B\to {\mathbb Q}_{\geq 0}$. Now, observe, first of all, that if $b\in B$ and $\widetilde{b}\in B_{\text{\tiny{red}}}$ is the  image of $b$ in $B/\text{Nil}(B)$,  then 
\begin{equation} \label{nub_reducido} 
\nub_{\mathfrak m}(b)= \nub_{{\mathfrak m}_{\text{\tiny{red}}}}(\widetilde{b}).
\end{equation}
Hence, in many situations we may reduce our proofs to the case in which the ring in consideration is reduced.
\medskip

\noindent
\textbf{Reduction to complete rings.}
Summing up, for a local ring $(B,\mathfrak{m})$,
let be $B_1$ as in (\ref{ExtEtale}) or as in (\ref{ExtBx}).
Then we have a chain of faithfully flat extensions,
\begin{equation} \label{ExtChain}
(B,\mathfrak{m})\to (B_1,\mathfrak{m}_1)\to 
(\widehat{B}_1,\widehat{\mathfrak{m}}_1).
\end{equation}
For any ideal $I\subset B$ and any $b\in B$,
the chain of equalities 
\begin{equation} \label{EqualChain}
\nub_{I}(b)=\nub_{I{B}_1}(b)= \nub_{I\widehat{B}_1}(b),
\end{equation}
is guaranteed (see \cite[Proposition 1.6.2]{Hu_Sw}).
In particular, if $I={\mathfrak m}$, then
\begin{equation} \label{MaxIdChain}
\nub_{\mathfrak m}(b)=\nub_{{\mathfrak m}{B}_1}(b)=\nub_{{\mathfrak m}_1}(b)=\nub_{{\mathfrak m}_1\widehat{B}_1}(b)=\nub_{\widehat{\mathfrak m}_1}(b).
\end{equation}
Thus, given an excellent, equidimensional, equicharacteristic local ring  $(B,\mathfrak{m})$, in most situations we will be able to reduce our proofs to the case of a complete reduced local ring containing either an infinite residue field or a field with sufficient scalars. By assuming that $B$ is excellent we will guarantee that $B$ is formally equidimensional and analytically unramified. The former condition allows us to use Rees' Theorem in Proposition \ref{ZarCond}, and the second will be implicitely used when reducing to the case of the completion of a reduced ring. 
\medskip

\noindent
\textbf{Good behaviour for equimultiple prime ideals.}
Let be ${\mathfrak p}\subset {\mathfrak m}$ a prime ideal in $B$
such that $e_{B_{\mathfrak p}}({\mathfrak p}B_{\mathfrak p})=e_B({\mathfrak m})$.
Then this condition is preserved if we consider some of the above extensions.
Set 
$(B',\mathfrak{m}')$ equal to either $B'=B_{\text{red}}$, or $B'=B_1$ as in
(\ref{ExtEtale}) or as in (\ref{ExtBx}), or $B'=\widehat{B}_1$ then
there exists a prime ideal $\mathfrak{p}'\subset B'$ dominating $\mathfrak{p}$, and
\begin{equation} \label{ExtPrime}
e_{B'_{\mathfrak p'}}({\mathfrak p'}B'_{\mathfrak p'})=e_{B'}({\mathfrak m'}).
\end{equation}
Moreover, if $\mathfrak{p}$ defines a regular subscheme in $\text{Spec}(B)$,
then $\mathfrak{p}'$ also defines a regular subscheme in $B'$.
\end{parrafo}
\medskip

\noindent{\bf {\em When does $\nub$ behave as a valuation?}}

\medskip

After the discussion in \S \ref{setting_1}, we get the following result:
\begin{proposition} \label{MiniMax}
	Let $(B,{\mathfrak m})$ be an equidimensional excellent equicharacteristic local ring of dimension $d\geq 1$. Suppose ${\mathfrak m}$ has a reduction generated by $d$ elements, $y_1,\ldots, y_d\in {\mathfrak m}$, and let $k\subset B$ be a field. Set $A:=k[y_1,\ldots, y_d]_{\langle y_1,\ldots, y_d\rangle}\subset B$.
Then, for $a\in A$ and $b\in B$, 
	$$\nub_{\mathfrak m}(ab)=\nub_{\mathfrak m}(a)+  \nub_{\mathfrak m}(b).$$
\end{proposition}
\begin{proof}
	Using the arguments in \S \ref{setting_1}, we can assume that $B$ is reduced and complete, and consider the finite-transversal extension 
$$S=k'[[y_1,\ldots, y_d]]\to B,$$
where $k'\supset k$ is a coefficient field of $B$. 
Now the result follows from \cite[Proposition 2.10]{BeBrEn}.
\end{proof}

The rest of the section	is devoted to the study of some more properties of finite-transversal projections.
\medskip

\noindent{\bf {\em On finite-transversal projections and the   top multiplicity locus of a ring}}

\medskip

\noindent The following three statements follow as a consequence 
of Proposition \ref{ZarCond} when applied to a local ring $(B,\mathfrak{m})$.
They are results  concerning   the primes  in $\Spec(B)$ that have the same multiplicity as that of $B$ at ${\mathfrak m}$, that is, the primes in the top multiplicity locus of $\Spec(B)$. 

\begin{proposition}\label{extension_transversalidad}
	Suppose  that  $(S, {\mathfrak n})\subset (B,{\mathfrak m})$ is   finite-transversal with respect to ${\mathfrak m}$. Let  ${\mathfrak p}\subset B$ be  a prime ideal with $e_{B_{\mathfrak p}}({\mathfrak p}B_{\mathfrak p})=e_B({\mathfrak m})$. Let ${\mathfrak q}={\mathfrak p}\cap S$.  Then:
	\begin{enumerate}
		\item[(i)] The extension $S_{\mathfrak q}\subset  B_{\mathfrak p}$ 
		is finite transversal with respect to ${\mathfrak p}$; 
		\item[(ii)] The local ring  $B/{\mathfrak p}$ is regular if and only if  $S/{\mathfrak q}$ is regular;
		\item[(iii)] If $B/{\mathfrak p}$ is regular then $S/{\mathfrak q}=B/{\mathfrak p}$. 
	\end{enumerate} 
\end{proposition}

\begin{proof}
Similar results were proven in \cite[Corollary 5.9, Proposition 6.3]{V} and \cite[Corollary 2.8]{COA} in the context of algebraic varieties defined over perfect fields. Here we check that the statement holds for more general rings under the hypotheses of the proposition. 
\medskip 
	
\noindent (i) Consider the following commutative diagram with vertical finite morphisms:
	$$\xymatrix{B \ar[r] &  B\otimes_S S_{\mathfrak q} \ar[r] & L\\
		S \ar[u] \ar[r] & S_{\mathfrak q} \ar[u]\ar[r]  &  K\ar[u].}$$
	Observe that the  generic rank of the extension $ S_{\mathfrak q}\to  B\otimes_S S_{\mathfrak q}$ is $m=e_{B_{\mathfrak p}}({\mathfrak p}B_{\mathfrak p})$. Hence, by definition,  $S_{\mathfrak q}\to  B\otimes_S S_{\mathfrak q}$  is finite-transversal with respect to ${\mathfrak p}$.  By Proposition \ref{ZarCond} (2) (i), 
	$B\otimes_S S_{\mathfrak q}=B_{\mathfrak p}$. In other words,   $S_{\mathfrak q}\subset  B_{\mathfrak p}$ 
	is finite transversal with respect to ${\mathfrak p}$. 
	\medskip
	
	\noindent (ii) Since $S_{\mathfrak q}\subset  B_{\mathfrak p}$ 
	is finite-transversal with respect to ${\mathfrak p}$, by Proposition  \ref{ZarCond} (2) (ii), $k({\mathfrak p})=k({\mathfrak q})$. Now consider the commutative diagram with vertical  finite extensions, 
	$$\xymatrix{B \ar[r] &  B/{\mathfrak p} \ar[r] & k({\mathfrak p})\\
		S \ar[u] \ar[r] & S/{\mathfrak q} \ar[u]\ar[r]  &  k({\mathfrak q}).\ar@2{-}[u]}$$
	Notice that  $S/{\mathfrak q}\to B/{\mathfrak p}$ is a finite extension  of local rings. Since conditions (2) (i)-(iii) of Proposition \ref{ZarCond} hold for $(S, {\mathfrak n})\to (B,{\mathfrak m})$, the same conditions hold for    $(S/{\mathfrak q},  {\mathfrak n}/{\mathfrak q})\to  (B/{\mathfrak p}, {\mathfrak m}/{\mathfrak p})$. Now   apply  Zariski's multiplicity formula for finite projections (Theorem \ref{ZarCond}) to $S/{\mathfrak q} \to  B/{\mathfrak p}$ to obtain, 
	$$1=e_{S/{\mathfrak q}}\cdot [k({\mathfrak p}): k({\mathfrak q})]= e_{B/{\mathfrak p}}\cdot  [k({\mathfrak m}): k({\mathfrak n}))]=e_{B/{\mathfrak p}},$$ 
	from where the claim in (ii) follows.
	
	\medskip
	
	\noindent (iii) By (ii) if $B/{\mathfrak p}$ is regular, then $S/{\mathfrak q}\subset B/{\mathfrak p} $ is a finite extension of regular local rings with the same quotient field. Since $S/{\mathfrak q}$ is regular, it is normal, and hence $S/{\mathfrak q}= B/{\mathfrak p}$.
\end{proof}

\medskip

\begin{proposition}\label{presentaciones_locales}
{ {\bf (Presentations of finite-transversal extensions)}}
	Suppose that  $(S, {\mathfrak n})\subset (B,{\mathfrak m})$ is finite-transversal with respect to ${\mathfrak m}$.  Let ${\mathfrak p}\subset B$ be a prime ideal with $e_{B_{\mathfrak p}}({\mathfrak p}B_{\mathfrak p})=e_B({\mathfrak m})$, and assume in addition that $B/{\mathfrak p}$ is a regular local ring. Let ${\mathfrak q}={\mathfrak p}\cap S$.  There there are $\theta_1,\ldots, \theta_e\in {\mathfrak p}$ such that: 
	\begin{enumerate}
		\item[(i)] $B=S[\theta_1,\ldots, \theta_e]$; 
		\item[(ii)] ${\mathfrak p}={\mathfrak q}B+\langle \theta_1,\ldots,\theta_e\rangle$.
\end{enumerate}
In addition,
\begin{enumerate}
		\item[(iii)]  ${\mathfrak q}B$ is a reduction of ${\mathfrak p}$ (in $B$). 
	\end{enumerate}
	
\end{proposition}
\begin{proof}
We follow ideas from \cite[Lemma 6.4]{V} for part (i), \cite[Lemma 8.10]{BeBrEn} for part (ii) and  \cite[Lemma 3.6]{COA} for part (iii), where similar results were proven in the context of algebraic varieties defined over perfect fields. To facilitate the reading of the paper we check here that the proofs can be adapted to cover a wider class of rings under the hypotheses of the proposition.
\medskip
	
\noindent (i) Write $B=S[\theta_1',\ldots, \theta_e']$. By Proposition \ref{extension_transversalidad} (iii), $S/{\mathfrak q}= B/{\mathfrak p}$, therefore, for each $i\in \{1,\ldots, e\}$, there is some $s_i\in S$ such that $\theta_i'-s_i\in {\mathfrak p}$. Set $\theta_i:=\theta_i'-s_i$ for $i=1\ldots, e$. Then $B=S[\theta_1,\ldots, \theta_e]$. 
	
	\medskip
	
\noindent (ii) By (i), we can write $B=S[\theta_1,\ldots, \theta_e]$ with $\theta_i\in {\mathfrak p}$ for $i=1,\ldots, e$.  Since  $S\subset B$ is finite-transversal at ${\mathfrak m}$, $S/{\mathfrak n}=B/{\mathfrak m}$, therefore, 
	\begin{equation}
		\label{escritura_maximal}
		{\mathfrak m}={\mathfrak n}+\langle \theta_1,\ldots, \theta_e\rangle.
	\end{equation}
	If ${\mathfrak p}={\mathfrak m}$ we are done. Otherwise,   since  ${\mathfrak q} \subset S$ defines a regular suscheme, there is a regular system of parameters in $S$, $y_1,\ldots, y_d$,  such that ${\mathfrak q}=\langle y_1,\ldots, y_r\rangle$ for some $r\in \{1,\ldots, d-1\}$. Then 
	\begin{equation}
		\label{contenido_primo}
		{\mathfrak q}B+\langle \theta_1,\ldots, \theta_e\rangle \subset {\mathfrak p}.
	\end{equation} 
	Now, 
	$$d-r=\dim(S/{\mathfrak q})=\dim (B/{\mathfrak p})\leq \dim (B/({\mathfrak q}B+\langle \theta_1,\ldots, \theta_e\rangle))\leq d-r,$$
	where the last inequality follows because  by  (\ref{escritura_maximal}),  
	$${\mathfrak m}/({\mathfrak q}+\langle \theta_1+\ldots+\theta_e\rangle)= ({\mathfrak n}+\langle \theta_1,\ldots, \theta_e\rangle)/ ({\mathfrak q}+\langle \theta_1+\ldots+\theta_e\rangle),$$
	and therefore, the maximal ideal ${\mathfrak m}/({\mathfrak q}+\langle \theta_1,\ldots, \theta_e\rangle)$ can be generated by $d-r$ elements. Since $B/{\mathfrak p}$ is an integral domain necessarily the containment in (\ref{contenido_primo}) is an equality. 
	
	\medskip
	
\noindent (iii) By (i) we can assume that $B=S[\theta_1,\ldots, \theta_e]$ with $\theta_i\in {\mathfrak p}$ for $i=1,\ldots, e$.  By Proposition \ref{extension_transversalidad} (i) and by Proposition \ref{ZarCond}, we have that ${\mathfrak q}B_{\mathfrak p}$ is a reduction of ${\mathfrak p}B_{\mathfrak p}$. To see that ${\mathfrak p}$ is the integral closure of ${\mathfrak q}B$ in $B$ it suffices to check this condition at all the maximal ideals containing ${\mathfrak p}$. Since $B$ is local, this amounts to checking this condition at $B$. 
	
	Since $B/{\mathfrak p}$ is a regular local ring, by Proposition \ref{extension_transversalidad} (iii), $S/{\mathfrak q}$ is also a regular local ring. Observe that the multiplicity of $S$ is 1, and so is the   multiplicity of  $S_{\mathfrak q}$.  Hence, by Theorem \ref{Hironaka_Schikhoff} below, $\text{ht}({\mathfrak q})=l({\mathfrak q})$ in $S$.  
	
	Now, since the extension $S\subset B$ is finite, $\text{ht}({\mathfrak q}B)=\text{ht}({\mathfrak q})$, and the blow up of $B$ at ${\mathfrak q}B$ is finite over the  blow up of $S$ at ${\mathfrak q}$. Hence, the fibers over the closed points have the same dimensions, and therefore, $l({\mathfrak q}B)=l({\mathfrak q})$. Therefore, $l({\mathfrak q}B)=\text{ht}({\mathfrak q}B)$ in $B$. Recall that  $e_{B_{\mathfrak p}}({\mathfrak q}B_{\mathfrak p})=e_{B_{\mathfrak p}}({\mathfrak p}B_{\mathfrak p})$,  by Proposition \ref{ZarCond} (iii). Finaly, since ${\mathfrak p}$ is the only minimal prime of ${\mathfrak q}B$, the statement follows from Theorem \ref{Boger} below. 

\end{proof}

\begin{theorem}
	\label{Hironaka_Schikhoff} (Hironaka--Schickhoff, \cite[Corollary 3, p. 121]{Lipman}) Let $(A,M)$ be a formally equidimensional local ring, and let ${\mathfrak p}\subset A$ be a prime ideal so that $A/{\mathfrak p}$ is regular. Then $\text{ht}({\mathfrak p})=l({\mathfrak p})$ in $A$ if and only if the local rings $A$ and $A_{\mathfrak p}$ have the same multiplicity. 
\end{theorem}

\begin{theorem}
	\label{Boger} (B\"oger, \cite[Theorems 2 and 3 p. 115-116]{Lipman}, \cite[Corollary 11.3.2]{Hu_Sw}) Let $(A,M)$ be a formally equidimensional local ring. Fix an ideal $I\subset A$ so that $\text{ht}(I)=l(I)$. Consider an ideal $J\subset A$ so that $I\subset J\subset \sqrt{I}$. Then $I$ is a reduction of $J$ if and only if $e_{A_{\mathfrak q}}(IA_{\mathfrak q})=e_{A_{\mathfrak q}}(JA_{\mathfrak q})$ for each minimal prime ideal ${\mathfrak q}$ of $I$. 
\end{theorem}

\begin{proposition}
	\label{transversal_intermedio} 
{\bf (Intermediate extensions of finite-transversal extensions)}
Suppose  that $(S, {\mathfrak n})\subset (B,{\mathfrak m})$ is   finite-transversal  with respect to ${\mathfrak m}$.  Let ${\mathfrak p}\subset B$ is a prime ideal with $e_{B_{\mathfrak p}}({\mathfrak p}B_{\mathfrak p})=e_B({\mathfrak m})$. Let $B'\subset B$ be an intermediate extension, i.e., $S\subset B'\subset B$,  and  consider the diagram:
	$$\xymatrix@C=1pt@R=1pc{ B  &   &  &  & {\mathfrak m} & {\mathfrak p}\\
		&  & B'  \ar[ull] & \ \  & {\mathfrak m}':={\mathfrak m}\cap B',  & {\mathfrak p}':={\mathfrak p}\cap B' \\
		S \ar[uu] \ar[urr] & & &  & {\mathfrak n}:={\mathfrak m}\cap S &    {\mathfrak q}:={\mathfrak p}\cap S.}$$
	Then:
	\begin{enumerate}
		\item[(i)] The ring $B'$ is local with maximal ideal ${\mathfrak m}'$; 
		\item[(ii)] The extension $S\subset B'$ is finite-transversal with respect to ${\mathfrak m}'$ of generic rank $m'=e_{B'}({\mathfrak m}')$; 
		\item[(iii)] The extension $S_{\mathfrak q}\subset B'_{{\mathfrak p}'}$ is finite-transversal with respect to ${\mathfrak p}'$ and  $e_{B'_{{\mathfrak p}'}}({\mathfrak p}'B'_{{\mathfrak p}'})=m'=e_{B'}({\mathfrak m}')$;
		\item[(iv)] For $b\in B'$, 
		\begin{equation}
			\label{igualdad_ordenes_maxl}
			\nub_{\mathfrak m}(b)=\nub_{{\mathfrak m}'}(b)
		\end{equation} 
		and 
		\begin{equation}
			\label{igualdad_ordenes_primos}
			\nub_{{\mathfrak p}B_{\mathfrak p}}(b)=\nub_{{\mathfrak p}'B_{{\mathfrak p}'}}(b);
		\end{equation} 
		\item[(v)] If $B/{\mathfrak p}$ is regular, then $B'/{\mathfrak p}'$ is regular, and in such case, ${\mathfrak q}B'$ is a reduction of  ${\mathfrak p}'$ in $B'$.  
	\end{enumerate}
\end{proposition}

\begin{proof}
 
	(i) This follows from the fact that $B$ is local and the extensions $S\subset B'\subset B$ are finite. 
	
	\medskip 
	
\noindent (ii) It suffices to check that the extension 
$S\subset B'$ satisfies conditions (2)(i)-(iii) of Proposition \ref{ZarCond}. Condition (2)(i) has already been proven, and condition 2(ii) follows from the chain of containments,  $$S/{\mathfrak n}\subset B'/{\mathfrak m}'\subset B/{\mathfrak m}=S/{\mathfrak n}.$$ To check that condition (2)(iii) holds, observe first  that ${\mathfrak n}B$ is a reduction of ${\mathfrak m}$, hence $\overline{{\mathfrak n}B}={\mathfrak m}$. On the other hand, 
	$${\mathfrak n}B\subset {\mathfrak m}'B\subset {\mathfrak m}=\overline{{\mathfrak n}B}.$$
	Since the extension $B'\subset B$ is finite, by \cite[Proposition 1.6.1]{Hu_Sw},  
	$$\overline{{\mathfrak n}B'}=\overline{{\mathfrak n}B}\cap B'={\mathfrak m}\cap B'={\mathfrak m}'.$$
\medskip	

\noindent (iii) By Proposition \ref{extension_transversalidad}(i), the extension $S_{\mathfrak q}\subset B_{\mathfrak p}$ is finite-transversal with respect to ${\mathfrak p}$. Repeating the argument in (ii) we find that $S_{\mathfrak q}\subset B'_{{\mathfrak p}'}$ is finite-transversal with respect to ${\mathfrak p}'$.
See also \cite[Lemma 4.12]{V} for (i), (ii) and (iii), in the case of domains.
\medskip
	
\noindent(iv) Equality (\ref{igualdad_ordenes_maxl})  follows from  (\ref{nub_integral}), from the fact that ${\mathfrak m}'B$ is a reduction of ${\mathfrak m}$, see \cite[Propositions 8.1.5 and 1.6.1]{Hu_Sw}.
Equality (\ref{igualdad_ordenes_primos}) follows similarly applying the previous argument to the finite-transversal extension $S_{\mathfrak q}\subset B'_{{\mathfrak p}'}\subset B_{\mathfrak p}$. 
\medskip

\noindent (v) By Proposition \ref{extension_transversalidad} (ii) and (iii)  $S/{\mathfrak q}$ is regular and, moreover $S/{\mathfrak q}=B/{\mathfrak p}$.
Then the first part of the statement follows.
Finally, by \cite[Proposition 1.6.1]{Hu_Sw} and Proposition \ref{presentaciones_locales}(iii),
we have $\overline{{\mathfrak q}B'}={\mathfrak p}'$.
\end{proof}

\begin{remark}
	\label{factorial} With the same hypotheses and notation as in Proposition \ref{transversal_intermedio}, observe  that if $B'=S[\theta]$ for some $\theta\in B$, then, by Proposition \ref{CoeffenS}, $S[\theta]\simeq S[Z]/\langle f(Z)\rangle$, where $f(Z)\in  S[Z]$ is the minimum polynomial of $\theta$ over $K$, the quotient field of $S$. The degree of $f(Z)$ is  the generic rank of $S\subset S[\theta]$, i.e., the dimension of the $K$-vector space $K\otimes_SS[\theta]$, which is bounded above by    $[L:K]=e_B({\mathfrak m})=m$. Therefore, by Theorem \ref{Hickel_refinado}, $\nub_{\mathfrak m}(\theta)\in \frac{1}{m!}{\mathbb N}$. See also \cite[Theorem 1.1]{Hickel13}. 
\end{remark}

\section{Some natural properties of the asymptotic Samuel function}
\label{seccion_naturalprops}

In this section we are going to explore some natural properties of the asymptotic Samuel function, addressing the proofs of the  results presented in the introduction. 

\begin{theorem}\label{casi_semi_continuo}
Let $B$ be an equidimensional excellent ring containing a field.
Let ${\mathfrak p}_1\subset{\mathfrak p}_2\subset B$ be two prime ideals such that   $e_{B_{\mathfrak{p}_1}}({\mathfrak p}_1B_{\mathfrak{p}_1})=
e_{B_{\mathfrak{p}_2}}({\mathfrak p}_2B_{\mathfrak{p}_2})$. Then $\nub_{{\mathfrak p}_1B_{\mathfrak{p}_1}}(b)\leq
\nub_{{\mathfrak p}_2B_{\mathfrak{p}_2}}(b)$ for $b\in B$.
\end{theorem}

\begin{proof} 
After localizing at $\mathfrak{p}_2$, we can assume that $(B,\mathfrak{m},k)$ is a local ring.
By the arguments in \S \ref{setting_1}, see  (\ref{ExtPrime}), we can start by assuming  that $B$ is reduced.  Consider the ${\mathfrak m}$-adic completion of $B$, $\widehat{B}$. Let   ${\mathfrak p}\subset \widehat{B}$ be a prime dominating ${\mathfrak p}_1B$. Then:
	$$\nub_{{\mathfrak p}_{1}B_{\mathfrak{p}_1}}(b)\leq \nub_{{\mathfrak p}_1\widehat{B}_{\mathfrak p}}(b) \leq \nub_{{\mathfrak p}\widehat{B}_{\mathfrak p}}(b).$$	
Hence, by the arguments detailed in  \S  \ref{setting_1}, to prove the theorem we can assume that $(B,{\mathfrak m})$ is a  reduced complete local ring  and that there is a finite-transversal projection 
$(S,\mathfrak{n})\to (B,\mathfrak{m})$, with $S=k[[y_1,\ldots,y_d]]$, where $k$ is the residue field of $B$, and $\langle y_1,\ldots, y_d\rangle$ generate a reduction of the maximal ideal ${\mathfrak m}$ of $B$. Let ${\mathfrak q}={\mathfrak p}\cap S$. 
	 By Proposition \ref{extension_transversalidad}, the extension $S_{\mathfrak q}\subset B_{\mathfrak p}$ is finite-transversal with respect to ${\mathfrak p}$. Now consider the diagram: 
	$$\xymatrix@R=1pc{B \ar[r] &   B_{\mathfrak p} \\
		B'=S[b] \ar[u] \ar[r] & B'_{{\mathfrak p}'}=S_{\mathfrak q}[b] \ar[u]\\
		S=k[[y_1,\ldots,y_d]] \ar[u] \ar[r] & S_{\mathfrak q}\ar[u]    }$$
	Let ${\mathfrak m}':={\mathfrak m}\cap S[b]$ and   ${\mathfrak p}':={\mathfrak p}\cap S[b]$. By Proposition \ref{transversal_intermedio}, the extension $S\subset S[b]$ is finite-transversal with respect to ${\mathfrak m}'$ of generic rank $\ell=e_{S[b]_{{\mathfrak m}'}}({\mathfrak m}')$ and $S_{\mathfrak q}\subset S_{\mathfrak q}[b]$ is finite transversal with respect to ${\mathfrak p}'$ with the same generic rank,  $\ell=e_{S[b]_{{\mathfrak m}'}}({\mathfrak m}')=e_{B'_{{\mathfrak p}'}}({\mathfrak p}')$. 
	
Using  Proposition \ref{transversal_intermedio} (iv), 
	$$\nub_{{\mathfrak p}B_{\mathfrak p}}(b)=\nub_{{{\mathfrak p}'}B'_{{\mathfrak p}'}}(b)  \  \  \ \text{ and } \  \  \nub_{\mathfrak m}(b)=\nub_{{\mathfrak m}'}(b),$$
	hence, it suffices to proof the theorem for $B'$. Now, by Proposition \ref{CoeffenS},   $S[b]\cong  S[Z]/\langle f(Z)\rangle $, where 
	$f(Z)$ is the minimal polynomial of $r$ over $S$. The degree of this polynomial equals the multiplicity of $S[b]$ at ${\mathfrak m}'$, $\ell$. By Theorem \ref{Hickel_refinado},   if 
	$$f(Z)=Z^{\ell}+a_1Z^{\ell -1}+\ldots+a_l,$$
	then 
	$$\nub_{{\mathfrak m}'}(b)=\min_i\left\{ \frac{\nu_{{\mathfrak  n}}(a_i)}{i}: i=1,\ldots, \ell  \right\}.$$
	Now, observe that $S[b]_{{\mathfrak p}'}=S_{{\mathfrak q}}[b]$, and therefore, again by Proposition \ref{CoeffenS},   $S[b]_{{\mathfrak p}'}=S_{{\mathfrak q}}[Z]\langle f(Z)\rangle$.  Hence, again by Theorem \ref{Hickel_refinado},  
	$$\nub_{{\mathfrak p}'B_{{\mathfrak p}'}'}(b)=\min_i\left\{ \frac{\nu_{{\mathfrak q}}(a_i)}{i}: i=1,\ldots, \ell  \right\}.$$ 
	To conclude by \cite[Theorem 2.11]{DDGH}, 
	for each $i\in \{1,\ldots, \ell\}$, 
	$ \frac{\nu_{{\mathfrak q}}(a_i)}{i}\leq  \frac{\nu_{{\mathfrak  n}}(a_i)}{i}$, 
	thus 
	$\nub_{{\mathfrak p}'B_{{\mathfrak p}'}'}(b)\leq \nub_{{\mathfrak m}'}(b)$. 
\end{proof}
\begin{remark}
    Observe that for a given $b\in B$ the function  
    $$\begin{array}{rrcl}
\nub(b): & \Spec(B) & \longrightarrow & \mathbb{Q}\cup\{\infty\} \\
 & \mathfrak{p} & \mapsto & \nub_{\p B_{\p}}(b)\end{array}$$
might not be upper semicontinuous, even after restricting ourselves to the top multiplicity locus of $B$. See Example \ref{Whitney}, where $\nub_{\p B_{\p}}(\overline{x})=1$, whereas $\nub_{\m}(\overline{x})=(p+1)/p$, for every maximal ideal $\m$ containing $\p$.
\end{remark}
\begin{theorem}\label{casi_semi_continuo_sin_localizar}
Let $B$ be an equidimensional excellent ring containing a field.
Let ${\mathfrak p}\subset B$ be a prime in the top multiplicity locus of $B$ and assume that $B/{\mathfrak p}$ is regular. Then
	$\nub_{\mathfrak p}(b)=\nub_{{\mathfrak p}B_{\mathfrak p}}(b)$ for $b\in B$. 
\end{theorem}
\begin{proof} 
Recall that if $n,\ell\in {\mathbb N}$, $\ell\neq 0$,  by (\ref{nub_integral}), $\nub_{{\mathfrak p}}(b)\geq n/\ell$ if and only if $b^{\ell}\in \overline{{\mathfrak p}^n}$. On the other hand, by  \cite[Proposition 1.1.4(4)]{Hu_Sw},
$\overline{{\mathfrak p}^n}B_{\mathfrak m}=\overline{{\mathfrak p}^nB_{\mathfrak m}}$ for all maximal ideals ${\mathfrak m}\subset B$. As a consequence,  $b^{\ell}\in \overline{{\mathfrak p}^n}$ if and only if, $b^{\ell}\in \overline{{\mathfrak p}^nB_{\mathfrak m}}$ for all maximal ideals ${\mathfrak m}\subset B$. Thus,  
  it suffices 
to prove that the  equality in the statement holds after localizing at each maximal ideal
$\mathfrak{m}\supset\mathfrak{p}$.
Hence we can assume that $(B,\mathfrak{m})$ is local and that
$e_{B_{\mathfrak p}}({\mathfrak p}B_{\mathfrak p})=e_B({\mathfrak m})$.
By the arguments in \S \ref{setting_1}, see (\ref{ExtPrime}), and the discussion at the beginning of the proof of Theorem \ref{casi_semi_continuo}, we can assume that $(B,{\mathfrak m})$ is reduced complete local ring and that there is a finite-transversal extension $S\subset B$. 
Let $p(z)=z^{\ell}+a_1z^{\ell-1}+\ldots+a_{\ell}\in S[z]$ be the minimal polynomial of $b$ over $S$, and let $\mathfrak{q}=\mathfrak{p}\cap S$.
Then, following the arguments in the first part of the proof of \cite[Theorem 2.1]{Hickel13}, $$\nub_{\mathfrak p}(b)\geq \min \left\{ \frac{\nu_{\mathfrak q}(a_i)}{i}: i=1,\ldots,m\right\}.$$
	 By  Proposition \ref{extension_transversalidad}, the prime  ${\mathfrak q}$ defines a regular prime in $\Spec(S)$. Hence, since 
  $S$ is regular and contains a field   we have that  the ordinary and symbolic powers of ${\mathfrak q}$ coincide. Therefore, 
	$$\nub_{\mathfrak p}(b)\geq \min \left\{ \frac{\nu_{\mathfrak q}(a_i)}{i}: i=1,\ldots,m\right\}=\min \left\{ \frac{\nu_{{\mathfrak q}S_{\mathfrak q}}(a_i)}{i}: i=1,\ldots,m\right\}=\nub_{\mathfrak pB_{\mathfrak p}}(b)\geq \nub_{\mathfrak p}(b).$$
\end{proof}
 
As indicated in the introduction, for a local ring   $(B, {\mathfrak m})$, the filtration $\{{\mathfrak m}^{\geq r}\}_{r\in {\mathbb Q}_{\geq 0}}$ leads us to the consideration of the graded ring $\overline{\Gr}_{\mathfrak m}(B)$, see \cite{LejeuneTeissier1974}.
Since $B$ is Noetherian, $\overline{\Gr}_{\mathfrak m}(B)$ is graded over the rationals with bounded denominators, i.e., there is some $\ell\in {\mathbb N}_{\geq 1}$ such that 
$\overline{\Gr}_{\mathfrak m}(B)=\bigoplus_{r\in \frac{1}{\ell}\mathbb{N}_{\geq 0}}{\mathfrak m}^{\geq r}/{\mathfrak m}^{> r}$. If in addition we impose that $B$ is excellent, reduced, equidimensional and equicharacteristic, then by Remark \ref{factorial}, $\ell$ can be taken as $m!$, where $m$ is the multiplicity of the local ring $B$. 

\begin{theorem}\label{graduado_barra_finito} 
{Let $(B,{\mathfrak m},k)$ be an excellent local ring. Then $\overline{\Gr}_{\mathfrak m}(B)$ is a $k$-algebra of finite type.} 
\end{theorem}
\begin{proof}  First observe that by (\ref{nub_reducido}) and from the definition of $\overline{\Gr}_{\mathfrak m}(B)$, we have that 
	$\overline{\Gr}_{\mathfrak m}(B)= \overline{\Gr}_{\mathfrak m}(B_{\text{red}})$, hence we may assume that $B$ is reduced. If $(B,{\mathfrak m})$ is regular there is nothing to prove. Otherwise,   
	 let $\ell:=m!$, where $m$ is the multiplicity of the local ring $B$. Let  
$${\mathcal G}:=B\oplus 0W^{\frac{1}{\ell}}\oplus \ldots \oplus  0W^{\frac{\ell-1}{\ell}}\oplus {\mathfrak m}W\oplus   0W^{\frac{\ell+1}{\ell}}\oplus \ldots \oplus  0W^{\frac{2\ell-1}{\ell}}\oplus {\mathfrak m}^2W^2\oplus \ldots,$$
i.e., ${\mathcal G}=\oplus_{n\in {\mathbb N}}{\mathfrak m}^{\frac{n}{\ell}}W^{\frac{n}{\ell}},$
where ${\mathfrak m}^0=B$, ${\mathfrak m}^{\frac{n}{\ell}}=(0)$ if $\frac{n}{\ell}\notin {\mathbb N}$, and $W$ is a variable that helps us keep track of the grading.  Define also, 
$${\mathcal H}:= B\oplus {\mathfrak m}^{ \geq\frac{1}{\ell}}W^{\frac{1}{\ell}}\oplus \ldots \oplus  {\mathfrak m}^{ \geq\frac{\ell-1}{\ell}}W^{\frac{\ell-1}{\ell}}\oplus {\mathfrak m}^{\geq 1}W\oplus   {\mathfrak m}^{ \geq\frac{\ell+1}{\ell}} W^{\frac{\ell+1}{\ell}}\oplus \ldots \oplus  {\mathfrak m}^{ \geq\frac{2\ell-1}{\ell}}W^{\frac{2\ell-1}{\ell}}\oplus {\mathfrak m}^{\geq 2}W^2\oplus \ldots,$$
i.e.,
${\mathcal H}=\oplus_{n\in {\mathbb N}}{\mathfrak m}^{\geq \frac{n}{\ell}}W^{\frac{n}{\ell}}$.

Then,  there is a containment of graded algebras ${\mathcal G}\subset {\mathcal H}$. Observe that  ${\mathcal G}$ is finitely generated over $B$ and that ${\mathcal H}$ is integral over ${\mathcal G}$, since, for a homogeneous element $fW^{\frac{n}{\ell}}\in {\mathcal H}$, we have that 
$$f^{\ell}\in \overline{{\mathfrak m}^n}.$$
Now,  $B$ is excellent,  and hence  so is ${\mathcal G}$. Let $L$ be the total quotient field of $B$. The  extension $L(W)\subset L(W^{\frac{1}{\ell}})$ is finite and  therefore  the  integral closure of ${\mathcal G}$ in $L(W^{\frac{1}{\ell}})$, $\overline{\mathcal G}$,  is finite over ${\mathcal G}$. Since 
${\mathcal G}\subset {\mathcal H}\subset \overline{\mathcal G}$, it follows that ${\mathcal H}$ is finite over ${\mathcal G}$, hence finitely generated over $B$. To conclude notice that 
$\overline{\Gr}_{\mathfrak m}(B)$ is a quotient of    ${\mathcal H}$. 
	\end{proof}

\section{Finiteness of the Samuel slope}
\label{seccion_finitud}

We devote the last sections of this paper to study properties of
the Samuel slope function defined in \S\ref{DefSamSlope}. Here we address the proofs of Theorem \ref{main_theorem} and Theorem \ref{ThSlopeReduced}.

In this section, and also in section \ref{SeccFFlatExt}, we will be using two main facts:
first, that the Samuel slope can be computed in the completion of the local ring
(Proposition \ref{Prop_Slope_etale}),
and second, if we are given a finite-transversal extension
$(S,\mathfrak{n})\to(B,\mathfrak{m})$ then, there is a procedure to approximate
the Samuel slope of $B$, using 
translations with elements in $S$ (Proposition \ref{base_ker}).

\begin{proposition} \label{Prop_Slope_etale} 
\cite[Proposition 3.10]{BeBrEn}
Let $(B,\mathfrak{m},k)$ be a Noetherian local ring.
\begin{itemize}
\item $(B,\mathfrak{m},k)\to(B',\mathfrak{m}',k')$ is an \'etale homomorphism
such that $k=k'$ then $\SSl(B)=\SSl(B')$.
\item If $(\hat{B},\widehat{\mathfrak{m}},k)$ denotes the
${\mathfrak m}$-adic completion of $B$ then 
$\SSl(B)=\SSl(\hat{B})$.
\end{itemize}
\end{proposition}

\begin{proof} 
The first assertion is Proposition 3.10 in \cite{BeBrEn}.
And the same proof applies to the completion, since it is enough to observe
that $\Gr_{\mathfrak{m}}(B)=\Gr_{\widehat{\mathfrak{m}}}(\hat{B})$.
\end{proof}

\begin{proposition} \cite[Lemma 8.9]{BeBrEn}
\label{base_ker}
Let $(S,\mathfrak{n})\to(B,\mathfrak{m})$ be a finite-transversal extension.
Write $B=S[\theta_1,\ldots, \theta_e]$ for some $\theta_i\in B$,
$i=1,\ldots,e$.
Set $d=\dim(S)=\dim(B)$. Suppose that the embedding dimension of $B$ is $d+t$, with $t>0$, and
that $\SSl(B)>1$.
Write  $\mathfrak{n}=\langle y_1,\ldots, y_d\rangle$ for some 
$y_i\in S$, $i=1,\ldots,d$.
Then, there are $s_i\in S$, $i=1,\ldots,e$ such that,
after reordering the elements $\theta_i$:
\begin{itemize}
\item[(i)]  $B=S[\theta'_1,\ldots, \theta'_e]$, where $\theta_i'=\theta_i+s_i$, and
\item[(ii)] $\{y_1,\ldots, y_d, \theta'_1, \ldots, \theta'_t\}$  is a minimal set of generators of $\mathfrak{m}$ with $t\leq e$. 
\end{itemize}
Furthermore, 
\begin{itemize} 
\item[(iii)] For a given  a $\lambda_{\mathfrak{m}}$-sequence,
$\{\delta_1,\ldots,\delta_t\}\subset B$, there are $s'_i\in S$, $i=1,\ldots,e$,
such that if $\theta''_i:=\theta_i+s'_i$ then,
\begin{itemize}
\item[(a)] $B=S[\theta_1'',\ldots, \theta_e'']$,
\item[(b)]
$\min \{\nub_{\mathfrak{m}}(\theta''_i)\mid i=1,\ldots, t, \ldots, e\}=
\min\{\nub_{\mathfrak{m}}(\theta''_i): i=1,\ldots, t\}\geq
\min \{\nub_{\mathfrak{m}}(\delta_i): i=1,\ldots, t\}$ and,
\item[(c)] $\{\theta_1'',\ldots, \theta_t''\}$ is a $\lambda_{\mathfrak{m}}$-sequence.
\end{itemize}
\end{itemize}	
\end{proposition}

\begin{corollary} \label{noentero}
Let $(B,\mathfrak{m})$ be
non-regular reduced equicharacteristic equidimensional excellent local ring.
Let $\{\theta_1,\ldots,\theta_t\}$ be a $\lambda_{\mathfrak{m}}$-sequence such that
$$\rho=\min\{\nub_{\mathfrak{m}}(\theta_1),\ldots,\nub_{\mathfrak{m}}(\theta_t)\}
\in\mathbb{Q}\setminus\mathbb{Z},$$
then $\SSl(B)=\rho$.
\end{corollary}

\begin{proof}
Without loss of generality we can assume that $\rho=\nub_{\mathfrak{m}}(\theta_1)$.
If $\SSl(B)>\rho$, then by Proposition \ref{base_ker}(iii), there exists some $s\in S$ such that $\overline{\nu}_{\mathfrak{m}}(\theta_1+s)>\overline{\nu}_{\mathfrak{m}}(\theta_1)$.
Since $\rho\in\mathbb{Q}\setminus\mathbb{Z}$,  observe that
$\nub_{\mathfrak{m}}(\theta_1)\neq \nub_{\mathfrak{m}}(s)$ for all $s\in S$ (since then $\nub_\mathfrak{m}(s)\in \mathbb{Z})$.
Hence if for some $s\in S$,
$\nub_{\mathfrak{m}}(\theta_1+s)\geq \nub_{\mathfrak{m}}(\theta_1)$ then
$\nub_{\mathfrak{m}}(\theta_1+s)=\nub_{\mathfrak{m}}(\theta_1)$,
 see Remark \ref{properties_order_function}.
\end{proof}

The following lemma, which was proven in \cite[Proposition 8.6]{BeBrEn} in the context of a local ring of an algebraic variety, is valid for Noetherian local rings with the same proof, which we briefly  sketch here.
We will use this result in the proof of Theorem \ref{main_theorem} below.

\begin{lemma}
\label{Lemadelacomplu}
Let $(B,\mathfrak{m},k)$ be a Noetherian local ring of dimension $d\geq 1$ which is in the extremal case. Then $B$ contains a reduction of $\mathfrak{m}$ generated by $d$ elements. 
\end{lemma}
\begin{proof} By \cite[Theorem 10.14]{H_I_O},  it suffices to find   $d$-elements $\kappa_1,\ldots, \kappa_d\in {\mathfrak m}\setminus {\mathfrak m}^2$ such that if $\overline{\kappa_1},\ldots, \overline{\kappa_d}$ denote  their images  in $ {\mathfrak m}_{\xi} /{\mathfrak m}_{\xi}^2 $, then $\text{Gr}_{{\mathfrak m}_{\xi}}({\mathcal O}_{X,\xi})/\langle \overline{\kappa_1},\ldots, \overline{\kappa_d} \rangle$  is a graded ring of dimension 0. Suppose 
	  $\dim_{k}{\mathfrak m}/{\mathfrak m}^2=d+t$. By  hypothesis $\dim_{k}\ker(\lambda)=t$, and if $\delta\in  {\mathfrak m}\setminus {\mathfrak m}^2$ is so that $\overline{\delta}\in \ker(\lambda_{\mathfrak m})$, then $\text{In}_{\mathfrak m}\delta\in \text{Gr}_{\mathfrak m}(B)$ is nilpotent. Thus, any collection of $d$-elements in ${\mathfrak m}\setminus {\mathfrak m}^2$ that completes a  $\lambda_{\mathfrak m}$-sequence to a basis of ${\mathfrak m}/{\mathfrak m}^2$ generates a reduction of ${\mathfrak m}$. 
\end{proof}

\begin{theorem}\label{main_theorem} Let $(B,{\mathfrak m},k)$ be a
non-regular reduced equicharacteristic equidimensional excellent local ring of dimension $d$.
Then $\SSl(B)\in\mathbb{Q}$.
\end{theorem}

\begin{proof}
First of all, if $B$ is not in the extremal case then $\SSl(B)=1$, and there is
nothing to prove.
Otherwise, $B$ is in the extremal case, and then it contains
a reduction of $\mathfrak{m}$ generated by $d$ elements
by Lemma \ref{Lemadelacomplu}.
By Proposition \ref{Prop_Slope_etale} we can assume that $B$ is a local complete ring, and
by \S\ref{setting_1} we can consider a finite-transversal extension
$(S,\mathfrak{n})\to(B,\mathfrak{m})$.
\medskip

Assume that the embedding dimension of $(B,\m)$ is $d+t$, where $t$ is the excess embedding dimension of $B$.
By Proposition \ref{base_ker}, we can write $B=S[\theta_1^{(0)},\ldots,\theta_e^{(0)}]$ 
with $\nub_{\mathfrak{m}}(\theta_j^{(0)})>1$, $j=1,\ldots,e$, and
$\mathfrak{m}=\mathfrak{n}B+\langle \theta_1^{(0)},\ldots,\theta_t^{(0)} \rangle$.
If $\min\{\nub_{\mathfrak{m}}(\theta_1^{(0)}),\ldots,\nub_{\mathfrak{m}}(\theta_t^{(0)})\}\in\mathbb{Q}\setminus\mathbb{Z}$ then by Corollary \ref{noentero}
we are done.
In fact, if there are some $s_i\in S$, $i=1,\ldots,t$, such that 
$\min\{\nub_{\mathfrak{m}}(\theta_1^{(0)}+s_1),\ldots,\nub_{\mathfrak{m}}(\theta_t^{(0)}+s_t)\}\in\mathbb{Q}\setminus\mathbb{Z}$ the result follows as well.
Therefore we can assume that $\min\{\nub_{\mathfrak{m}}(\theta_1^{(0)}+s_1),\ldots,\nub_{\mathfrak{m}}(\theta_t^{(0)}+s_t)\}\in\mathbb{Z}$ for every 
$s_i\in S$, $i=1,\ldots,t$.
Hence the only way that the $\SSl(B)\not\in\mathbb{Q}$ is that $\SSl(B)=\infty$.
Suppose  that we can find  a sequence of $\lambda_{\m}$-sequences $\{\{\gamma^{(i)}_1,\dots,\gamma^{(i)}_t\}\}_{i\geq 0}$ such that $\min\{\overline{\nu}_{\m}(\gamma^{(i)}_1),\dots,\overline{\nu}_{\m}(\gamma^{(i)}_t)\}$ tends to infinity as $i$ grows. Using Proposition \ref{base_ker} we can find a sequence $\{\{\theta_1^{(i)},\dots,\theta_t^{(i)}\}\}_{i\geq 0}$ such that:
\begin{enumerate}
    \item [(i)] There exists some $s_j^{(i)}\in S$ such that $\theta_j^{(i)}=\theta_j^{(0)}+s_{j}^{(i)}$, for every $j=1,\dots,t$ and $i\geq 1$, 
    \item [(ii)] $\mathfrak{m}=\mathfrak{n}B+\langle \theta_1^{(i)},\dots,\theta_t^{(i)}\rangle$,
    \item[(iii)] $\min\{\overline{\nu}_{\m}(\theta_j^{(i)})\,:\, j=1,\dots,t\}\geq \min\{\overline{\nu}_{\m}(\gamma_j^{(i)})\,:\,j=1,\dots,t\}$,
    \item[(iv)] the set $\{\theta^{(i)}_1,\dots,\theta^{(i)}_t\}$ forms a $\lambda_{\m}-$sequence.
\end{enumerate}
Note that by the condition in (i), $\nub_{\mathfrak{m}}(s_j^{(i)})>1$ and
combining this with condition (iii) it follows that 
\begin{equation} \label{EqInTheta}
0\neq \In(\theta_j^{(i)})=\In(\theta_j^{(0)})\in\mathfrak{m}/\mathfrak{m}^2.
\end{equation}

Taking a subsequence if necessary, we may assume that the sequence $\{\overline{\nu}_{\m}(\theta_1^{(i)})\}_{i\geq 1}$ is strictly monotonically increasing (to those purposes note that for every $i\geq 0$,
$\overline{\nu}_{\m}(\theta_1^{(i)})\neq\infty$ since $B$ is reduced by hypothesis). It follows then that the sequence $\{s_1^{(i)}\}_{i\geq 1}$ is a Cauchy sequence, since 
$$\overline{\nu}_{\m}(\theta_1^{(i)})=\overline{\nu}_{\m}(s_1^{(i+1)}-s_1^{(i)})=\nu_{\mathfrak{n}}(s_1^{(i+1)}-s_1^{(i)})\in\mathbb{Z}_{\geq 1},$$
where the last equality is a consequence of \cite[Proposition 2.10]{BeBrEn}. Since $S$ is complete, the sequence $\{s_1^{(i)}\}_{i\geq 0}$ converges in $S$ to an element $s_1$.
Hence, the element $\theta=\theta^{(0)}_1+s_1$ is nonzero
(see (\ref{EqInTheta})) and it is the limit of the sequence $\{\theta^{(i)}_1\}_{i\geq 1}=\{\theta^{(0)}_1+s_{1}^{(i)}\}_{i\geq 1}$ satisfying that $\overline{\nu}_{\m}(\theta)=\infty$,
contradicting the fact that $B$ is reduced.
\end{proof}

\begin{corollary} \label{TrasladaS}
Let $(B,\mathfrak{m})$ be
non-regular reduced equicharacteristic equidimensional excellent local ring.
Assume that there is a finite-transversal projection $(S,\mathfrak{n})\to (B,\mathfrak{m})$,
and that $B=S[\theta_1,\ldots,\theta_e]$ for some
$\theta_1,\ldots,\theta_e\in B$.
Then there exist $s_1,\ldots,s_e\in S$ such that
$$\SSl(B)=\min\{\nub_{\mathfrak{m}}(\theta_i-s_i)\mid i=1,\ldots,e\}.$$
\end{corollary}

\begin{proof}
This is a direct consequence of Proposition \ref{base_ker} 
and Theorem \ref{main_theorem}.
\end{proof}

\begin{theorem} \label{ThSlopeReduced}
Let $(B,{\mathfrak m},k)$ be an equicharacteristic equidimensional excellent local ring. Then 
$\SSl(B)=\SSl(B_{\text{\tiny{red}}})$.
\end{theorem}

\begin{proof} 

The natural surjective morphism of local rings, 
$(B,{\mathfrak m},k) \to  (B_{\text{\tiny{red}}},{\mathfrak m}_{\text{\tiny{red}}},k)$ 
induces a surjective linear map of $k$-vector spaces,
$h: {\mathfrak m}/{\mathfrak m}^2 \to {\mathfrak m}_{\text{\tiny{red}}}/{\mathfrak m}_{\text{\tiny{red}}}^2$,
with $\ker(h)=(\text{Nil}(B)+{\mathfrak m}^2)/{\mathfrak m}^2$, and a commutative diagram of $k$-vector spaces, 
$$\xymatrix{
{\mathfrak m}/{\mathfrak m}^2  \ar[r]^{h} \ar[d]_{\lambda_{\mathfrak m}} &
{\mathfrak m}_{\text{\tiny{red}}}/{\mathfrak m}_{\text{\tiny{red}}}^2 \ar[d]^{\lambda_{{\mathfrak m}_{\text{\tiny{red}}}}}\\
{\mathfrak m}^{(\geq 1)}/{\mathfrak m}^{(> 1)}  \ar[r]^{h'}  &
{\mathfrak m}_{\text{\tiny{red}}}^{(\geq 1)}/{\mathfrak m}_{\text{\tiny{red}}}^{(> 1)}.
}$$
Note that $h'$ is an isomorphism.

There are linear subspaces $L\subset{\mathfrak m}/{\mathfrak m}^2$ and 
$L_{\text{\tiny{red}}}\subset{\mathfrak m}_{\text{\tiny{red}}}/{\mathfrak m}_{\text{\tiny{red}}}^2$ such that
$${\mathfrak m}/{\mathfrak m}^2=
L\oplus \ker(\lambda_{{\mathfrak m}}),
\qquad
{\mathfrak m}_{\text{\tiny{red}}}/{\mathfrak m}^2_{\text{\tiny{red}}}=
L_{\text{\tiny{red}}}\oplus \ker(\lambda_{{\mathfrak m}_{\text{\tiny{red}}}}).$$
Since $\ker(h)\subset\ker(\lambda_{{\mathfrak m}})$, then $L\cong L_{\text{\tiny{red}}}$ and
there exist a linear subspace
$H\subset\ker(\lambda_{{\mathfrak m}})$ such that 
$\ker(\lambda_{{\mathfrak m}})=H\oplus\ker(h)$.
Hence $H\cong\ker(\lambda_{{\mathfrak m}_{\text{\tiny{red}}}})$ via $h$.

If $B_{\text{\tiny{red}}}$ is a regular local ring then $t_{\text{\tiny{red}}}=0$
and $\SSl(B_{\text{\tiny{red}}})=\infty$ by definition.
In this case $\ker(\lambda_{\mathfrak m})=\ker(h)$, and hence
for any $\lambda_{\mathfrak{m}}$-sequence $\theta_1,\ldots,\theta_t$,
we have that $\theta_i\in \text{Nil}(B)$ for all $i=1,\ldots,t$.
The result follows since $\nub_{\mathfrak{m}}(\theta_i)=\infty$.

Assume now that $B_{\text{\tiny{red}}}$ is non-regular,
then $t_{\text{\tiny{red}}}>0$ and
by Theorem \ref{main_theorem} $\SSl(B_{\text{\tiny{red}}})<\infty$.
If $\SSl(B_{\text{\tiny{red}}})=1$ then 
$\dim_k(L)=\dim_k(L_{\text{\tiny{red}}})>d=\dim(B)$, and we have that $\SSl(B)=1$.

If $\SSl(B_{\text{\tiny{red}}})>1$ then the result follows from (\ref{nub_reducido})
and because 
every $\lambda_{{\mathfrak{m}}_{\text{\tiny{red}}}}$-sequence
can be extended to a $\lambda_{\mathfrak{m}}$-sequence with elements in $\ker(h)$,
and, reciprocally, every $\lambda_{\mathfrak{m}}$-sequence contains a
$\lambda_{{\mathfrak{m}}_{\text{\tiny{red}}}}$-sequence.
\end{proof}

\begin{corollary}
\label{Pendienteinfinita}
    Let $(B,\m,k)$ be an equicharacteristic equidimensional excellent local ring. Then $\SSl(B)=\infty$ if and only if $B_{\text{\tiny{red}}}$ is regular.
\end{corollary}

\section{The Samuel slope after some faithfully flat extensions}
\label{SeccFFlatExt}

In this section we will show that the Samuel slope of a local
ring remains the same after the faithfully flat extensions considered in
\S\ref{setting_1}.

\begin{proposition} \label{LemaResField}
Let $(B,\mathfrak{m})$ be a Noetherian local ring. Set $B'=B[x]_{\mathfrak{m}[x]}$
and let $\mathfrak{m}'=\mathfrak{m}B'$ be the maximal ideal of $B'$.
Then
$\SSl(B)=\SSl(B')$.
\end{proposition}

\begin{proof}
If $B$ is regular, there is nothing to prove.
Otherwise, observe that the excess of embedding dimension of $B$, $t$, is the same
as that of $B'$.
If $B$ is not in the extremal case, then $B'$ is not in the extremal case either,
and then $\SSl(B)=1=\SSl(B')$.
Therefore it remains to prove the statement if $B$ is in the extremal case,
and hence so is $B'$.  
The inequality $\SSl(B)\leq\SSl(B')$ is straightforward.
Let us prove that $\SSl(B)\geq\SSl(B')$.

Every element $\theta'\in B'$ can be expressed, up to a unit, as a polynomial
$$\theta'=\theta_{0}+\theta_{1}x+\cdots+\theta_{r}x^r,$$
for some $r\in\mathbb{N}$ and where $\theta_i\in B$.
Note that 
$$\nub_{\mathfrak{m}'}(\theta')=
\min\{\nub_{\mathfrak{m}}(\theta_{0}),\ldots,\nub_{\mathfrak{m}}(\theta_{r})\}.$$
This follows from the fact that
$$\nub_{\mathfrak{m}}(\theta_i)\geq\frac{a}{b}
\Longleftrightarrow \theta_i^b\in\overline{\mathfrak{m}^a},
\qquad
\nub_{\mathfrak{m}'}(\theta')\geq\frac{a}{b}
\Longleftrightarrow \theta'^b\in\overline{{\mathfrak{m}'}^a},$$
and $\overline{{\mathfrak{m}'}^a}=\overline{\mathfrak{m}^a}B'$,
see \cite[Lemma 8.4.2(9)]{Hu_Sw}. 

Assume that $\theta'_1,\ldots,\theta'_t$ is a $\lambda_{\mathfrak{m}'}$-sequence.
Up to some units in $B'$, every $\theta'_i$ can be expressed as a polynomial
$$\theta'_i=\theta_{i,0}+\theta_{i,1}x+\cdots+\theta_{i,r_i}x^{r_i},
\qquad i=1,\ldots,t,$$
where $\theta_{i,j}\in B$.
We may assume that $\theta_{i,0}\not\in\mathfrak{m}^2$ for all $i=1,\ldots,t$. 

We have that
\begin{equation} \label{EqTerminosCte}
\min\{\nub_{\mathfrak{m}'}(\theta'_1),\ldots,\nub_{\mathfrak{m}'}(\theta'_t)\}
\leq
\min\{\nub_{\mathfrak{m}}(\theta_{1,0}),\ldots,\nub_{\mathfrak{m}}(\theta_{t,0})\}.
\end{equation}
If the classes of $\theta_{1,0},\ldots,\theta_{t,0}$ in $\mathfrak{m}/\mathfrak{m}^2$
are linearly independent then $\{\theta_{1,0},\ldots,\theta_{t,0}\}$ is
a $\lambda_{\mathfrak{m}}$-sequence and we conclude that $\SSl(B)\geq\SSl(B')$.

If the classes of $\theta_{1,0},\ldots,\theta_{t,0}$ are linearly dependent, 
there are $\mu_1,\ldots,\mu_t\in B$, not all zero in $B/\mathfrak{m}$, such that
$$\mu_1\theta_{1,0}+\cdots+\mu_t\theta_{t,0}\in\mathfrak{m}^2.$$
Since $\theta_{i,0}\not\in\mathfrak{m}^2$  for $i=1,\dots,t$, there are at least two indices
$i\neq j$ such that $\mu_i,\mu_j\not\in\mathfrak{m}$.
Let $i_0$ be such that
$$\nub_{\mathfrak{m}'}(\theta'_{i_0})=
\min\{\nub_{\mathfrak{m}'}(\theta'_1),\ldots,\nub_{\mathfrak{m}'}(\theta'_t)\}.$$
Then either $i_0\neq i$ or $i_0\neq j$.
Assume that $i_0\neq i$ and define 
$$\theta''_{\ell}:=\theta'_{\ell}, \quad \ell\neq i
\qquad\text{and} \qquad
\theta''_i:=x^{-1}\left(\mu_1\theta'_1+\cdots+\mu_t\theta'_t-
\left(\mu_1\theta_{1,0}+\cdots+\mu_t\theta_{t,0}\right)\right).$$
We have that $\theta''_1,\ldots,\theta''_t$ is a $\lambda_{\mathfrak{m}'}$-sequence
and
$$\min\{\nub_{\mathfrak{m}'}(\theta'_1),\ldots,\nub_{\mathfrak{m}'}(\theta'_t)\}=
\min\{\nub_{\mathfrak{m}'}(\theta''_1),\ldots,\nub_{\mathfrak{m}'}(\theta''_t)\},$$
so that inequality (\ref{EqTerminosCte}) also holds and
the degree of the polynomial for $\theta''_i$ is smaller.
After finitely many steps we arrive to the case where 
$\theta_{1,0},\ldots,\theta_{t,0}$ form a $\lambda_{\mathfrak{m}}$-sequence.
\end{proof}

\begin{definition} \cite[Definition 4.4]{BVIndiana}
Let $S$ be a regular ring and let $\mathfrak{q}\subset S$ be a  prime such that the quotient $S/\q$ is a regular ring.
Let $f(z)\in S[z]$ be a monic polynomial of degree $m$ in $z$:
$$f(z)=z^m+a_1z^{m-1}+\cdots+a_{m-1}z+a_m, \quad a_i\in S,\ j=1,\ldots,m.$$
Set $r_j=\nu_{\mathfrak q}(a_j)$ for $j=1\ldots m$,
and set
$$q:=\min\left\{\dfrac{r_j}{j}: j=1,\ldots, m\right\}=
\min\left\{\dfrac{\nu_{\mathfrak q}(a_j)}{j}: j=1,\ldots, m\right\}.$$
For every $j=1,\ldots,m$, if $jq=r_j$ then set
$A_j:=\In_{\mathfrak{q}}(a_j)\in \mathfrak{q}^{jq}/\mathfrak{q}^{jq+1}$,
and if $jq<r_j$, set $A_j:=0$.

We define the \emph{weighted initial form} of $f$ at $\mathfrak{q}$
as the polynomial: 
\begin{equation} \label{EqWeightIn}
\wwin_{\mathfrak{q}}(f(z)):=z^m+\sum_{j=1}^{m}A_jz^{m-j}\in\Gr_{\mathfrak{q}}(S)[z],
\end{equation}
where $\Gr_{\mathfrak{q}}(S)=\oplus_{i\geq 0}\mathfrak{q}^i/\mathfrak{q}^{i+1}$.
Note that $\wwin_{\mathfrak{q}}(f(z))$ is a weighted polynomial of degree $mq$, where
the degree of $z$ is $q$ and the degree of elements in $\mathfrak{q}/\mathfrak{q}^2$
is one.
\end{definition}

\begin{remark} \label{MejoraWwin}
Let $(S,\mathfrak{n})\to (B,\mathfrak{m})$ be a finite-transversal projection.
Let $\theta\in B$, and set $q=\nub_{\mathfrak{m}}(\theta)$.
If $f(z)=z^m+a_1z^{m-1}+\cdots+a_{m-1}z+a_m\in S[z]$ is the minimal polynomial of $\theta$ over $S$, we can associate
to $\theta$ the weighted initial form of $f$ at $\mathfrak{n}=\mathfrak{m}\cap S$,
$\wwin_{\mathfrak{n}}(f)$.
Note that $\wwin_{\mathfrak{n}}(f(z))$ is a monic polynomial on $z$ of degree $m$ different from $z^m$ since $B$ is reduced.
In particular, there is some $j$ with $A_j\neq 0$.

In fact, there exists some $s\in S$ such that
$\nub_{\mathfrak{m}}(\theta-s)>\nub_{\mathfrak{m}}(\theta)$ if and only if
$\wwin_{\mathfrak{n}}(f)$ is an $m$-th power.
See \cite[Remark 4.6]{BVIndiana} for a discussion on the context of algebraic varieties.
\end{remark}

\begin{remark} \label{RemWwin}
Let $(S,\mathfrak{n})\to (B,\mathfrak{m})$ be a finite-transversal projection
of equicharacteristic local rings.
Assume that $B=S[\theta]$ for some $\theta\in B$.
By Corollary \ref{TrasladaS} there exists some $s\in S$ such that $\nub_{\mathfrak{m}}(\theta-s)=\SSl(B)$.
Remark \ref{MejoraWwin} gives us an iterative procedure to find $s\in S$:

If $\wwin_{\mathfrak{n}}(f(z))$ is not an $m$-th power then $\nub_{\mathfrak{m}}(\theta)=\SSl(B)$.

If $\wwin_{\mathfrak{n}}(f(z))$ is an $m$-th power then choose $s_1\in S$ such that
$\wwin_{\mathfrak{n}}(f(z))=(z-\In_{\mathfrak{n}}(s_1))^m$.
Note that in this case $\nub_{\mathfrak{m}}(\theta)$ must be an integer.
Set $\theta_1=\theta-s_1$. We know that
$\nub_{\mathfrak{m}}(\theta_1)>\nub_{\mathfrak{m}}(\theta)$,
and then we can repeat the procedure with $\theta_1$.
Observe that $f_1(z)=f(z+s)$ is the minimal polynomial of $\theta_1$.

Now, since $\SSl(B)<\infty$ (Theorem \ref{main_theorem}), it is clear that,
after finitely many steps, the weighted initial 
form is not an $m$-th power.
\end{remark}

\begin{theorem} \label{ThSlopeEtale}
Let $(B,{\mathfrak m},k)$ be an  equicharacteristic equidimensional excellent local ring,
and let  $(B,{\mathfrak m})\to (B',{\mathfrak m}')$ be a local-\'etale extension.
Then $\SSl(B)=\SSl(B')$. 
\end{theorem}

\begin{proof}
{If $B$ is regular or if  $\SSl(B)=1$ there is nothing to prove.
	Otherwise let $k=B/\mathfrak{m}$ and let $k'=B'/\mathfrak{m}'$.
	If $k=k'$ then the result is \cite[Proposition 3.10]{BeBrEn}.
	After ruling out the previous cases, by Theorem \ref{ThSlopeReduced}, we can assume that both $B$ and $B'$ are reduced, in the extremal case, and that $k\subsetneq k'$ is a separable finite extension. 	Set $d=\dim(B)$. Since $B$ is in the extremal case, by Lemma \ref{Lemadelacomplu}, ${\mathfrak m}$ has a reduction generated by $d$ elements, $\langle y_1,\ldots, y_d\rangle$. Observe that $\langle y_1,\ldots, y_d\rangle$ also expands to a reduction of ${\mathfrak m}'$ in $B'$, and similarly, expands to reductions of $\widehat{\mathfrak m}$ in $\widehat{B}$ and 
 $\widehat{{\mathfrak m}'}$ in $\widehat{B'}$ respectively. 	
	
	Since $B\to B'$ is \'etale, by \cite[Proposition 17.6.3]{EGAIV}, $\widehat{B'}$ is formally \'etale over  $\widehat{B}$ for the ${\mathfrak m}$-adic topologies, furthermore,  $\widehat{B'}$ is faithfully flat over $\widehat{B}$ and finite. To ease notation,  let us denote again by  $k$ some coefficient field of $\widehat{B}$. By the natural map $\widehat{B}\to \widehat{B'}$, the image of $k$ maps into some coeficient field of $\widehat{B'}$ which for simplicity we denote by $k'$. Then we have the following commutative diagram: 
	$$\xymatrix{\widehat{B} \ar[r] & \widehat{B'} \\ 
		S=k[[y_1,\ldots,y_d]] \ar[r]\ar[u] & S'=k'[[y_1,\ldots,y_d]],\ar[u].}$$
	where the lower horizontal map is local \'etale,  the vertical maps are finite-transversal, and we write $y_1,\ldots, y_d$ for the images of these elements in both $\widehat{B}$ and $\widehat{B'}$.

	Write $\widehat{B}=S[\theta_1,\ldots,\theta_e]$ for some $\theta_i\in B$
	(see \ref{setting_1}).
	Then it can be checked that $\widehat{B'}=S'[\theta_1,\ldots,\theta_e]$ 
	(here we use the fact that $\widehat{B'}$ is finite over $\widehat{B}$).
	
	Denote by $\mathfrak{n}$ (resp. $\mathfrak{n}'$) the maximal ideal of $S$
	(resp. of $S'$).
	For $i=1,\ldots, e$, let $\wwin_{\mathfrak{n}}(f_i)$ be the 
	weighted initial form of $f_i(z_i)$, the minimal polynomial of $\theta_i$ over $S$.
	Note that $f_i(z_i)$ is also the minimal polynomial of $\theta_i$ over $S'$.
	To justify this consider the following diagram
	$$\xymatrix@R=1pc{\widehat{B} \ar[r] & \widehat{B'} \\ 
		S[\theta_i] \ar[r]\ar[u] & S'[\theta_i]\ar[u] \\
		S \ar[r]\ar[u] & S'.\ar[u]}$$
	By Proposition \ref{CoeffenS}, $S[\theta_i]\cong S[z_i]/f_i(z_i)$,
	and $S'[\theta_i]=S[\theta_i]\otimes_S S'$.
	
	The image of $\wwin_{\mathfrak{n}}(f_i)\in \Gr_{\mathfrak{n}}(S)$ in 
	$\Gr_{\mathfrak{n}'}(S')$ is $\wwin_{\mathfrak{n}'}(f_i)$.
	Now we conclude, since 
	$\wwin_{\mathfrak{n}}(f_i)$ is an $m$-power in $\Gr_{\mathfrak{n}}(S)$
	if and only if
	$\wwin_{\mathfrak{n}'}(f_i)$ is an $m$-power in $\Gr_{\mathfrak{n}'}(S')$.
	  Here we are using the fact that
	$\Gr_{\mathfrak{n}'}(S')=\Gr_{\mathfrak{n}}(S)\otimes_k k'$ and
	the extension $k\to k'$ is \'etale (see \cite[Proposition 16.2.2]{EGAIV}, in fact flatness is enough to guarantee the isomorphism). Now, 
	the result follows from Remark \ref{MejoraWwin} and Corollary \ref{TrasladaS},  because: $$\SSl(\widehat{B})=\min\{\SSl(S[\theta_i])
	\mid i=1,\ldots,e\}.$$ }
\end{proof}

\section{Comparing slopes at prime ideals}
\label{seccion_comparing slopes at prime ideals}

As indicated in the Introduction, for a Noetherian ring $B$, the function
$$\begin{array}{rrcl}
\SSl: & \Spec(B) & \longrightarrow & \mathbb{Q}\cup\{\infty\} \\
 & \mathfrak{p} & \mapsto & \SSl(B_{\mathfrak{p}})
\end{array}$$
is not upper semicontinuous in general,
even after restricting to the top multiplicity locus of $B$.
This can be checked in the following example:

\begin{example}
\label{Whitney}
Let  $p\in {\mathbb Z}_{>0}$ be a prime number, and let
$B:={\mathbb F}_p[x,y_1,y_2]/\langle f\rangle$  
where $f=x^p-y_1^p y_2$.
Observe that ${\mathfrak p}=\langle \overline{x},\overline{y_1}\rangle$ determines a non-closed point
in $\Spec(B)$ of maximum multiplicity $p$.
It can be checked that $\SSl(B_{\mathfrak{p}})=\nub_{\p B_{\p}}(\overline{x})=1$.
However for every maximal ideal $\mathfrak{m}\supset\mathfrak{p}$ we have that
$\SSl(B_{\mathfrak{m}})=\nub_{\m}(\overline{x})=(p+1)/p$.
\end{example}

Observe that in the example $\SSl(B_{\mathfrak p})\leq\SSl(B_{\mathfrak{m}})$ for all maximal
ideals $\mathfrak{m}\supset\mathfrak{p}$.
In fact this will happen quite generally, as the following result states.

\begin{theorem} \label{ThSlopeSemicont}
Let $B$ be an equidimensional excellent ring containing a field and let
$\mathfrak{p}\in\Spec(B)$.
Then there is a dense open set $U\subset\MaxSpec(B/\mathfrak{p})$ such that
$$\SSl(B_{\mathfrak p})\leq\SSl(B_{\mathfrak{m}})\qquad
\text{for all }\quad \mathfrak{m}/\mathfrak{p}\in U.$$
\end{theorem}

Before addressing the proof of the theorem we need an auxiliary result.

\begin{proposition} \label{PrimoRegTras}
Let $(S,\mathfrak{n})\to (B,\mathfrak{m})$ be a finite-transversal projection
of equicharacteristic local rings.
Suppose that $B=S[\theta]$ for some $\theta\in B$.
Let $\mathfrak{p}$ be a prime in $B$ such that $B/\mathfrak{p}$ is regular, and
$m=e_{B_{\mathfrak{p}}}(\p B_{\p})=e_B(\m)>1$.
Then there is some $s\in S$ such that:
$$\nub_{\mathfrak{p}}(\theta-s)=\nub_{\mathfrak{p}B_{\mathfrak{p}}}(\theta-s)
=\SSl(B_{\mathfrak{p}}).$$
\end{proposition}

\begin{proof}
Set $\mathfrak{q}=\mathfrak{p}\cap S$, and
let $f(z)\in S[z]$ be the minimal polynomial of $\theta$ over $S$,
$$f(z)=z^m+a_1z^{m-1}+\cdots+a_{m-1}z+a_m, \qquad a_i\in S.$$
By Proposition \ref{CoeffenS}(1), $f(z)$ is also the minimal polynomial of $\theta$
over $S_{\mathfrak{q}}$.
By Proposition \ref{CoeffenS}(2), $B=S[\theta]\cong S[z]/\langle f(z)\rangle$,
therefore the generic rank of the finite-transversal extension $S\subset B$
is $m=e_B(\m)$.
Since $e_{B_{\mathfrak{p}}}(\p B_{\p})=m$, and $S_{q}\subset B_{\mathfrak{p}}$ is 
finite-transversal (Proposition \ref{extension_transversalidad}(i))
the generic rank is also $m$ and also 
$B_{\mathfrak{p}}=S_{\mathfrak{q}}[\theta]\cong S_{\mathfrak{q}}[z]/\langle f(z)\rangle$.

By Corollary \ref{TrasladaS}, there is some $\tilde{s}\in S_{\mathfrak{q}}$ with
$\nub_{\mathfrak{p}B_{\mathfrak{p}}}(\theta-\tilde{s})
=\SSl(B_{\mathfrak{p}})$.
Remark \ref{RemWwin} indicates that $\tilde{s}$ can be obtained looking at the weighted initial
form $\wwin_{\mathfrak{q}S_{\mathfrak{q}}}(f(z))$.

Since $\mathfrak{p}$ is a regular prime in $B$, $\mathfrak{q}$ is a regular prime
in $S$ (see Proposition \ref{extension_transversalidad}(ii)). 
Consider the natural map
\begin{equation} \label{EqGrprimo}
\Gr_{\mathfrak{q}}(S)=\bigoplus_{i\geq 0}\mathfrak{q}^i/\mathfrak{q}^{i+1}
\to \Gr_{\mathfrak{q}S_{\mathfrak{q}}}(S_{\mathfrak{q}})=
\Gr_{\mathfrak{q}}(S)\otimes_{S/{\mathfrak{q}}} K(S/\mathfrak{q}).
\end{equation}
Note that $\wwin_{\mathfrak{q}S_{\mathfrak{q}}}(f(z))$ is the image of
$\wwin_{\mathfrak{q}}(f(z))$ by the map in (\ref{EqGrprimo}).
Then both rings in (\ref{EqGrprimo}) are regular, in particular they are
UFDs and the second is a localization of the first. Now it follows that 
$\wwin_{\mathfrak{q}}(f(z))$ is a $m$-th power if and only if
$\wwin_{\mathfrak{q}S_{\mathfrak{q}}}(f(z))$ is a $m$-th power.
Hence there is some ${s}$ in $S$ such that
$$\nub_{\mathfrak{p}}(\theta-s)=\nub_{\mathfrak{p}B_{\mathfrak{p}}}(\theta-s)
=\SSl(B_{\mathfrak{p}}).$$
\end{proof}

\begin{proof}[Proof of Theorem \ref{ThSlopeSemicont}]
Since $B$ is excellent, by \cite[Theorem 2.33]{CJS2020}, there exists a dense open set $U$ in
$\Spec(B/\mathfrak{p})$ such that $e_{B_{\mathfrak{p}}}(\p B_{\p})=e_{B_{\m}}(\m B_{\m})$
for every $\mathfrak{m}/\mathfrak{p}\in U$.
After shrinking $U$ if needed, we can also assume that $B/\mathfrak{p}$ is regular at all
maximal ideals in $U$.
Hence we can assume to be in the case where $B$ is the localization at some maximal ideal
$\mathfrak{m}/\mathfrak{p}\in U$.

By Proposition \ref{LemaResField} we may assume that the residue field of $B$ is infinite,
and by Theorem \ref{ThSlopeReduced} we may assume that $B$ is reduced.
By Proposition \ref{Prop_Slope_etale} and the arguments in \S\ref{setting_1}
we may assume that $B$ is complete: here we use the fact that 
$\SSl(B_{\p})\leq \SSl(\hat{B}_{\p\hat{B}})$ and $\SSl(\hat{B})=\SSl(B)$,
hence it suffices to prove that $\SSl(\hat{B}_{\p\hat{B}})\leq \SSl(\hat{B})$.

Again, by the arguments if \S\ref{setting_1} we have a finite-transversal
extension $S\subset B$.
Set ${\mathfrak q}:=\mathfrak{p}\cap S$. 
By Proposition \ref{extension_transversalidad} we have that
$S_{\mathfrak{q}}\to{B}_{\mathfrak{p}}$ is finite-transversal and
$S/\mathfrak{q}=B/\mathfrak{p}$.

By Proposition \ref{presentaciones_locales},
$\mathfrak{q}B$ is a reduction of $\mathfrak{p}\subset B$ and there are 
$\theta_1,\ldots,\theta_e\in\mathfrak{p}$ such that
$B=S[\theta_1,\ldots,\theta_e]$, 
$${\mathfrak m}={\mathfrak n}B+\langle \theta_1,\ldots,\theta_e\rangle,
\qquad\text{and}\qquad 
{\mathfrak p}={\mathfrak q}+\langle \theta_1,\ldots,\theta_e\rangle.$$

Now consider the commutative diagram:
$$\xymatrix@R=1pc{B=S[\theta_1,\ldots, \theta_e] \ar[r] &  S_{\mathfrak q}[\theta_1,\ldots,\theta_e]=B_{\mathfrak{p}} \\
 	S[\theta_i] \ar[u] \ar[r] & S_{\mathfrak q}[\theta_i] \ar[u]\\
 	S=k[|y_1,\ldots,y_d|] \ar[u] \ar[r] & S_{\mathfrak q}.\ar[u]    }$$
Set $\mathfrak{m}_i=\mathfrak{m}\cap S[\theta_i]$ and 
$\mathfrak{p}_i=\mathfrak{p}\cap S[\theta_i]$, for $i=1,\ldots,e$.

By Proposition \ref{PrimoRegTras}, there are some $s_i\in\mathfrak{p}$ such that 
for $i=1\ldots,e$,
$$\nub_{\mathfrak{p}}(\theta_i-s_i)=\nub_{\mathfrak{p}_i}(\theta_i-s_i)
=\SSl(S_{\mathfrak{q}}[\theta_i]).$$
Therefore, after translation by elements of $S$, we can assume that $\nub_{\p}(\theta_i)=\SSl(S_{\q}[\theta_i])$. On the other hand, note that
$$\SSl(B_{\mathfrak{p}})=\min\{\SSl(S_{\mathfrak{q}}[\theta_i])
\mid i=1,\ldots,e\}.$$
Now the result follows since for $i=1,\ldots,e$ we have:
$$\SSl(S[\theta_i])\geq\nub_{\mathfrak{m}}(\theta_i)\geq
\nub_{\mathfrak{p}}(\theta_i)
=\SSl(S_{\mathfrak{q}}[\theta_i])$$
and 
$$\SSl(B)=\min\{\SSl(S[\theta_i])
\mid i=1,\ldots,e\}.$$
\end{proof}

\noindent
\textbf{Competing interest.}
The authors have no competing interests to declare that are relevant to the content of this article.

\end{document}